\documentclass[11pt]{article}

\usepackage{amsmath, amsthm, amssymb, tikz}
\usepackage{amsfonts}
\usepackage{fullpage}
\usepackage[enableskew]{youngtab}

\newcommand\beq{\begin{equation}}
\newcommand\eeq{\end{equation}}
\newcommand\bce{\begin{center}}
\newcommand\ece{\end{center}}
\newcommand\bea{\begin{eqnarray}}
\newcommand\eea{\end{eqnarray}}
\newcommand\ba{\begin{array}}
\newcommand\ea{\end{array}}
\newcommand\ben{\begin{enumerate}}
\newcommand\een{\end{enumerate}}
\newcommand\bit{\begin{itemize}}
\newcommand\eit{\end{itemize}}
\newcommand\brr{\begin{array}}
\newcommand\err{\end{array}}
\newcommand\bt{\begin{tabular}}
\newcommand\et{\end{tabular}}

\newcommand\nn{\nonumber}

\newcommand\ms{\medskip}
\newcommand\ul{\underline}
\newcommand\ol{\overline}
\renewcommand\S{{\mathcal S}}
\newcommand\s{{\tt S}}
\newcommand\U{{\mathcal U}}
\newcommand\A{{\mathcal A}}
\newcommand\LL{{\mathcal L}}
\newcommand\Sh{{\mathcal F}}
\newcommand\Z{{\mathcal Z}}
\newcommand\T{{\mathcal T}}
\newcommand\G{{\mathcal G}}
\newcommand\Hk{{\mathcal H}}
\DeclareMathOperator\maj{maj} 
\DeclareMathOperator\RSK{RSK}
\newcommand\x{{\mathbf x}}
\newcommand\ten{10}
\newcommand\eleven{11}
\newcommand\twelve{12}
\newcommand\blank{$ $}
\newcommand\hs{\hspace{2.6mm}}
\newcommand{\W}{{\rm{Weak}}}
\newcommand{\shape}{{\rm{shape}}}
\newcommand\ddom{{d_{\rm{dom}}}}
\newcommand\dXn{{d_{X_n}}}

\newcommand{\bbz}{\mathbb{Z}}

\newcommand{\tC}{{\widetilde{C}}}

\newcommand{\D}{{\rm{Des}}}

\newcommand{\invol}{\varphi}
\newcommand{\WW}{{\mathcal W}}
\newcommand{\ww}{\mathbf{w}}
\newcommand{\enc}{\nu}

\newcommand{\ZZ}{\mathbb{Z}}

\newcommand{\CC}{\mathbb{C}}

\newcommand{\inv}{\operatorname{inv}}

\newtheorem{theorem}{Theorem}[section]
\newtheorem{proposition}[theorem]{Proposition}
\newtheorem{lemma}[theorem]{Lemma}
\newtheorem{corollary}[theorem]{Corollary}
\newtheorem{definition}[theorem]{Definition}
\newenvironment{ex}[1][Example.]{\begin{trivlist}
\item[\hskip \labelsep {\bfseries
#1}]}{\end{trivlist}}

\newtheorem{question}[theorem]{Question}

\newtheorem{remark}[theorem]{Remark}

\newtheorem{claim}[theorem]{Claim}

\numberwithin{figure}{section}

\title{Arc Permutations}

\author{Sergi Elizalde~\thanks{Department of Mathematics, Dartmouth College, Hanover, NH 03755, USA. {\tt sergi.elizalde@dartmouth.edu}.
Partially supported by NSF grant DMS-1001046.} \and Yuval
Roichman~\thanks{Department of Mathematics, Bar-Ilan University,
 Ramat-Gan 52900, Israel.  {\tt yuvalr@math.biu.ac.il}. Partially supported
 by Bar-Ilan Rector Internal Research Grant.}}

\date{}

\begin{document}

\maketitle

\begin{abstract}
Arc permutations and unimodal permutations were introduced in the
study of triangulations and characters. This paper studies combinatorial properties
and structures on these permutations. First, both sets are characterized by pattern avoidance.
It is also shown that arc permutations carry a natural affine Weyl
group action, and that the number of geodesics between a
distinguished pair of antipodes  in the associated Schreier graph,
as well as the number of maximal chains in the weak order on
unimodal permutations, are both equal to twice the number of
standard Young tableaux of shifted staircase shape. Finally, a
bijection from non-unimodal arc permutations to Young tableaux of
certain shapes, which preserves the descent set, is described and
applied to deduce a conjectured character formula of Regev.
%\keywords{Arc permutation \and Unimodal permutation \and Pattern avoidance \and Affine Weyl group \and Descent set \and Shifted staircase \and Weak order}

%\subclass{05A05 \and 05E18 \and 05E10 \and 05A19 \and 05A15}
\end{abstract}

\tableofcontents

\section{Introduction}

A permutation in the symmetric group $\S_n$ is an {\em arc
permutation} if every prefix forms an interval in $\bbz_n$. It was
found recently that arc permutations play an important role in the
study of graphs of triangulations of a polygon~\cite{TFT2}.
A familiar subset of arc permutations is that of unimodal arc permutations,
which are the permutations whose inverses have one local
maximum or one local minimum.
These permutations appear in the study of Hecke algebra
characters~\cite{Gelfand,RIMS}. Their cycle structure was studied by Thibon~\cite{Thibon} and others.

In this paper we study
combinatorial properties and structures on these sets of
permutations.

In Section~\ref{pattern-section} it is shown that both arc and
unimodal permutations may be characterized by pattern avoidance,
as described in Theorem~\ref{arc-pattern-thm} and
Proposition~\ref{unimodal-pattern}.

In Section~\ref{section-prefixes} we describe a bijection between unimodal permutations and certain shifted shapes.
The shifted shape corresponding to a unimodal permutation $\pi$ has the property that standard Young tableaux of that shape encode
all reduced words of $\pi$. It follows that
\begin{itemize}
\item Domination in the weak order  on unimodal permutations is
characterized by inclusion of the corresponding shapes
(Theorem~\ref{weak_domination}). Hence, this partially ordered set is a modular
lattice (Proposition~\ref{weak-properties}). \item The number of
maximal chains in this order is equal to twice the number of
staircase shifted Young tableaux, that is, $2\binom{n}{2}! \cdot
\prod_{i=0}^{n-2} \frac{i!}{(2i+1)!}$
(Corollary~\ref{mc_enumeration}).
\end{itemize}

The above formula is analogous to a well-known result of Richard
Stanley~\cite{Stanley-mc}, stating that the number of maximal
chains in the weak order on $\S_n$ is equal to the number of
standard Young tableaux of triangular shape.

In Section~\ref{X-graph} we study a graph on arc permutations,
where adjacency is defined by multiplication by a simple
reflection. It is shown that this graph has the following
property: an arc permutation is unimodal if and only if it
appears in a geodesic between two distinguished antipodes. Hence
the number of geodesics between these antipodes is, again,
$2\binom{n}{2}! \cdot \prod_{i=0}^{n-2} \frac{i!}{(2i+1)!}$. This
result is analogous to \cite[Theorem 9.9]{TFT2}, and related
to~\cite[Theorem 2]{Panova}.

The set of non-unimodal arc permutations is not a union of Knuth
classes. However, it carries surprising Knuth-like properties,
which are described in Section~\ref{equid}. A bijection between
non-unimodal arc permutations and standard Young tableaux of hook
shapes plus one box is presented, and shown to preserve the
descent set. This implies that for $n\ge4$,
$$\sum_{T\in\T_n}\x^{\D(T)}=\sum_{\pi\in\Z_n}\x^{\D(\pi)},$$
where $\Z_n$ denotes the set of non-unimodal arc permutations in
$\S_n$, $\T_n$ denotes the set of standard Young tableaux of shape
$(k,2,1^{n-k-2})$ for some $2\le k\le n-2$, and $\D(\pi)$ is the
descent set of $\pi$ (see Theorem~\ref{thm:TZ}).
Further enumerative results on arc permutations by descent sets appear in Section~\ref{sec:encoding}.
These enumerative results are then applied to prove a conjectured
character formula of Amitai Regev in Section~\ref{sec:regev}.

Interactions with other mathematical objects are discussed in the last
two sections: close relations to shuffle permutations are pointed
out in Section~\ref{sec:appendix}; further representation theoretic
aspects are discussed in Section~\ref{sec:remarks}.
In particular, Section~\ref{section-action} studies
a transitive affine Weyl group action on the set of arc
permutations, whose resulting Schreier graph is the graph studied in Section~\ref{X-graph}.

\section{Basic concepts}

In the following definitions, an interval in $\bbz$ is a subset $[a,b]=\{a,a+1,\dots,b\}$ for some $a\le b$,
and an interval in $\bbz_n$ is a subset of the form $[a,b]$ or $[b,n]\cup[1,a]$ for some $1\le a\le b\le n$.

\subsection{Unimodal permutations}

\begin{definition} A permutation $\pi\in \S_n$ is  {\em left-unimodal}
if, for every $1\le j\le n$, the first $j$ letters in $\pi$ form
an interval in $\bbz$. Denote by $\LL_n$ the set of left-unimodal permutations in $\S_n$.
\end{definition}

\begin{ex} The permutation $342561$ is left-unimodal, but $3412$ is not. \end{ex}

\begin{claim}\label{claimL}
$|\LL_n|=2^{n-1}$.
\end{claim}

\begin{proof}
A left-unimodal permutation $\pi$ is uniquely determined by the subset of values $i\in\{2,\dots,n\}$ such that $\pi(i)>\pi(1)$. There are $2^{n-1}$ such subsets.
\end{proof}

We denote by $\D(\pi)$ the descent set of a permutation $\pi$, and
by $\RSK(\pi)=(P,Q)$ the pair of standard Young tableaux
associated to $\pi$ by the RSK correspondence. For a standard Young tableau $T$, its descent set $\D(T)$ is defined as the set of entries $i$ that lie strictly above the row where $i+1$ lies.
It is well known that if $\RSK(\pi)=(P,Q)$, then $\D(\pi)=\D(Q)$ and $\D(\pi^{-1})=\D(P)$.

\begin{remark}\label{knuth1}\rm
A permutation $\pi$ is left-unimodal if and only if
$\D(\pi^{-1})=\{1,2,\dots,i\}$ for some $0\le i\le n-1$.
In other words $\pi\in \LL_n$ if and only if
$\RSK(\pi)=(P,Q)$, where $P$ is a hook with entries 
$1,2,\dots,i+1$ in the first column, and $Q$ is any hook with the
same shape as $P$.
It follows that left-unimodal permutations are a union of Knuth
classes.
\end{remark}

\begin{definition} A permutation $\pi\in \S_n$ is  {\em unimodal}
if one of the following holds: \begin{itemize} \item[(i)] every
prefix forms an interval in $\bbz$; or \item[(ii)]every suffix forms
an interval in $\bbz$.
\end{itemize}
Denote by $\U_n$ the set of unimodal permutations in $\S_n$.
\end{definition}

We remark that our definition of unimodal permutations is slightly different from the one given in~\cite{Gelfand,RIMS}, where unimodal permutations are those whose inverse is left-unimodal in this paper, and in~\cite{Thibon}, where unimodal permutations are those whose inverse is right-unimodal in our terminology.

\begin{ex} The permutation $165243$ is unimodal. \end{ex}

\begin{claim}\label{claimU}
For $n\ge2$, $|\U_n|=2^{n}-2$.
\end{claim}

\begin{proof}
A permutation $\pi\in\S_n$ is unimodal if either $\pi$ or its
reversal $\pi^R=\pi(n)\dots\pi(2)\pi(1)$ is left-unimodal. The
only permutations for which both $\pi$ and $\pi^R$ are
left-unimodal are $12\dots n$ and $n\dots 21$. The formula now
follows from Claim~\ref{claimL}.
\end{proof}

\begin{remark}\label{knuth2}\rm
 A permutation $\pi\in\S_n$ is unimodal if and only if
$$\D(\pi^{-1})=\begin{cases} \{1,2,\dots,i\} \mbox{ or}\\ \{i+1,i+2,\dots,n-1\} \end{cases}$$
for some $1\le i\le n-1$. This happens if and only if
$\RSK(\pi)=(P,Q)$, where $P$ is a hook with entries
$1,2,\dots,i+1$ in the first column or in the first row, and $Q$
is any hook with the same shape as $P$. Thus unimodal permutations
are a union of Knuth classes.
\end{remark}

\subsection{Arc permutations}

\begin{definition} A permutation $\pi\in \S_n$ is an {\em arc permutation}
if, for every $1\le j\le n$, the first $j$ letters in $\pi$ form
an interval in $\bbz_n$ (where
the letter $n$ is identified with zero). Denote by $\A_n$ the set of arc permutations in $\S_n$.
\end{definition}

\begin{ex} The permutation $12543$ is an arc permutation
in $\S_5$, but $125436$ is not an arc permutation in
$\S_6$, since $\{1,2,5\}$ is an interval in $\bbz_5$ but not in
$\bbz_6$.
\end{ex}

\begin{claim}\label{arc-enumeration}
For $n\ge2$, $|\A_n|=n2^{n-2}$.
\end{claim}

\begin{proof}
To build $\pi\in\A_n$, there are $n$ choices for $\pi(1)$ and two choices for every other
letter except the last one.
\end{proof}

\begin{remark}\label{knuth3}\rm
Arc permutations are not a union of Knuth classes. Note, however,
that arc permutations may be characterized in terms of descent
sets as follows. A permutation $\pi\in\S_n$ is an arc permutation
if and only if
$$\D(\pi^{-1})=\begin{cases} \{1,2,\dots,i,j+1,j+2,\dots,n-1\} \mbox{ and }\pi^{-1}(1)<\pi^{-1}(n), \quad \mbox{or}\\ \{i+1,i+2,\dots,j\} \mbox{ and }\pi^{-1}(1)>\pi^{-1}(n)\end{cases}$$
for some $i\le j$.
\end{remark}

It is clear from the definition that the sets of left-unimodal,
unimodal and arc permutations satisfy $\LL_n\subset \U_n\subset
\A_n$. We denote by $\Z_n=\A_n\setminus \U_n$ the set of
non-unimodal arc permutations. It follows from
Remarks~\ref{knuth2} and~\ref{knuth3} that $\Z_n$ is not a union
of Knuth classes. However, $\Z_n$ has some surprising Knuth-like
properties, which will be described in Section~\ref{equid}.

\section{Characterization by pattern
avoidance}\label{pattern-section}

In this section the sets of left-unimodal permutations, arc permutations, and unimodal permutations are characterized in
terms of pattern avoidance. Given a set of patterns $\tau_1,\tau_2,\dots$, denote by $\S_n(\tau_1,\tau_2,\dots)$ the set of permutations in $\S_n$ that avoid
all of the $\tau_i$, that is, that do not contain a subsequence whose entries are in the same relative order as those of $\tau_i$.  Define $\A_n(\tau_1,\tau_2,\dots)$ analogously.

\subsection{Left-unimodal permutations}

It will be convenient to use terminology from {\em geometric
grid classes}. Studied by Albert et al.~\cite{AABRV}, a geometric
grid class consists of those permutations that can be drawn on a
specified set of line segments of slope $\pm1$, whose locations are
determined by the positions of the corresponding entries in a matrix $M$ with entries in $\{0,1,-1\}$. 
More precisely, $\G(M)$ is the set of permutations that can be obtained by
placing $n$ dots on the segments in such a way that there are no two dots on the same vertical or horizontal line, labeling the dots with $1,2,\dots,n$ by increasing $y$-coordinate, and then reading
them by increasing $x$-coordinate.  All the geometric
grid classes that we consider in this paper are also profile classes in the sense of Murphy and Vatter~\cite{MV}.

Left-unimodal permutations are those that can be drawn on the picture on the left of Figure~\ref{fig:grid_staircase}, which consists of a segment of slope $1$ above a segment of slope $-1$.
The picture on the right shows a drawing of the permutation $32415$.  The grid class of permutations that can be drawn on this picture is denoted by $$\G\left(\ba{c} 1\\ -1 \ea\right),$$
so we have that $$\LL_n=\G_n\left(\ba{c} 1\\ -1 \ea\right)=\G\left(\ba{c} 1\\ -1 \ea\right)\cap\S_n.$$

\begin{figure}[htb]
\begin{center}
\includegraphics[width=3cm,angle=-90]{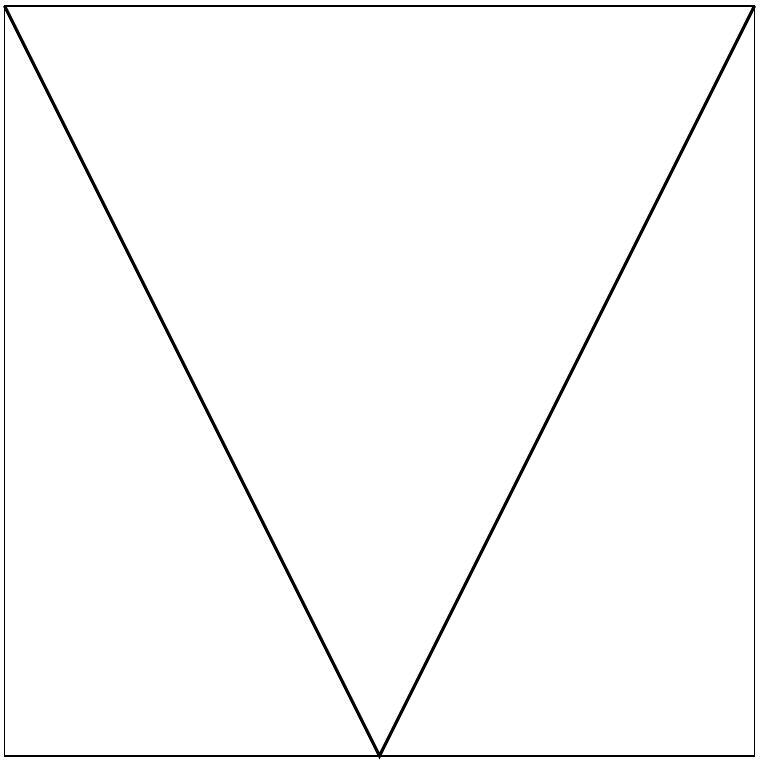}\hspace{2cm}\includegraphics[width=3cm,angle=-90]{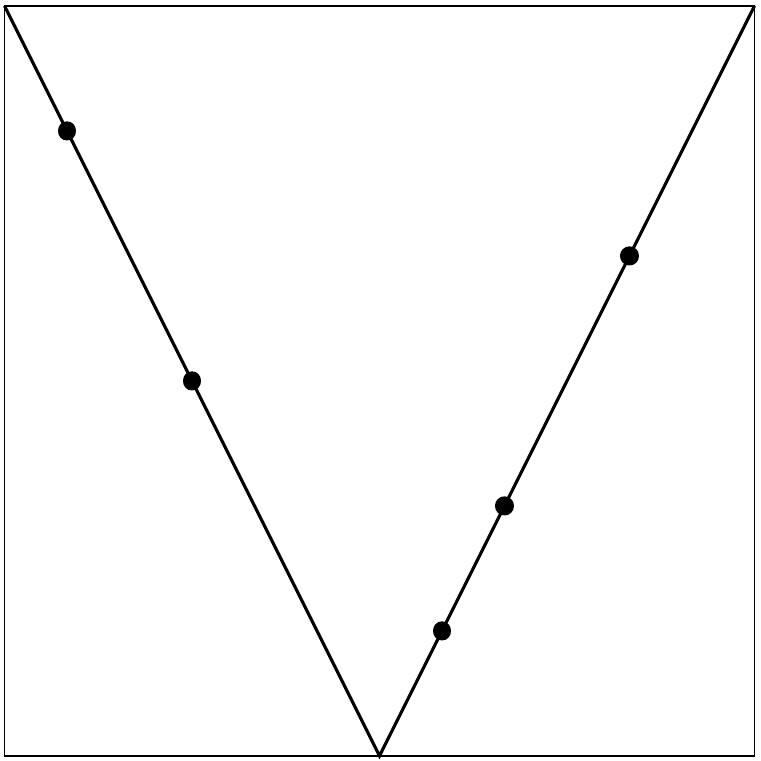}
\caption{\label{fig:grid_staircase} The grid for left-unimodal permutations, and a drawing of the permutation $32415$.}
\end{center}
\end{figure}

It is clear from the description that geometric grid
classes are always closed under pattern containment, so they are characterized by the set of minimal forbidden patterns. In the case of left-unimodal permutations, we get the following description.

\begin{claim}
$\LL_n=\S_n(132,312)$.
\end{claim}

\begin{proof}
The condition that every prefix of $\pi$ is an interval in $\bbz$ is equivalent to the condition that there is no
pattern $\pi(i)\pi(j)\pi(k)$ (with $i<j<k$) where the value of $\pi(k)$ is between
$\pi(i)$ and $\pi(j)$, that is, $\pi$ avoids $132$ and
$312$.
\end{proof}

\subsection{Arc permutations}

Arc permutations can be characterized in terms of pattern avoidance, as those permutations avoiding the eight patterns $\tau\in\S_4$
with $|\tau(1)-\tau(2)|=2$.

\begin{theorem}\label{arc-pattern-thm}
$$\A_n=\S_n(1324,1342,2413,2431,3124,3142,4213,4231).$$
\end{theorem}

\begin{proof}
For an integer $m$, denote by $\ol{m}$ the element of
$\{1,2,\dots,n\}$ that is congruent with $m$ mod $n$.
Let $\pi\in\S_n$, and suppose that $\pi\notin\A_n$. Let $i>1$ be
the smallest number with the property that $\{\pi(1),\pi(2),\dots,\pi(i)\}$
is not an interval in $\bbz_n$. By minimality of $i$, the set
$\{\pi(1),\pi(2),\dots,\pi(i-1)\}$ contains neither $\ol{\pi(i)+1}$
nor $\ol{\pi(i)-1}$. Letting $j<k$ be such that
$\{\pi(j),\pi(k)\}=\{\ol{\pi(i)+1},\ol{\pi(i)-1}\}$, it follows
that $\pi(i-1)\pi(i)\pi(j)\pi(k)$ is an occurrence of one of the
eight patterns above.

Conversely, if $\pi\in\S_n$ contains one of the eight patterns,
let $\pi(h)\pi(i)\pi(j)\pi(k)$ be such an occurrence, where $h<i<j<k$.
Then $\{\pi(1),\pi(2),\dots,\pi(i)\}$ is not an interval in $\bbz_n$.

\end{proof}

\begin{corollary}
 $|\S_n(1324,1342,2413,2431,3124,3142,4213,4231)|=n2^{n-2}$ for $n\ge2$.
\end{corollary}

Arc permutations can also be described in terms of grid classes, as those permutations that can be
drawn on one of the two pictures in Figure~\ref{fig:grid_arc}. We write
$$\A_n=\G_n\left(\begin{array}{cc}
           1 & 0 \\
           -1 & 0 \\
           0 & -1 \\
           0 & 1 \\
         \end{array}
       \right)\ \cup\
\G_n\left(\begin{array}{cc}
           0 & -1 \\
           0 & 1 \\
           1 & 0 \\
           -1 & 0 \\
         \end{array}
       \right).$$

\begin{figure}[htb]
\begin{center}
\includegraphics[width=4cm,angle=-90]{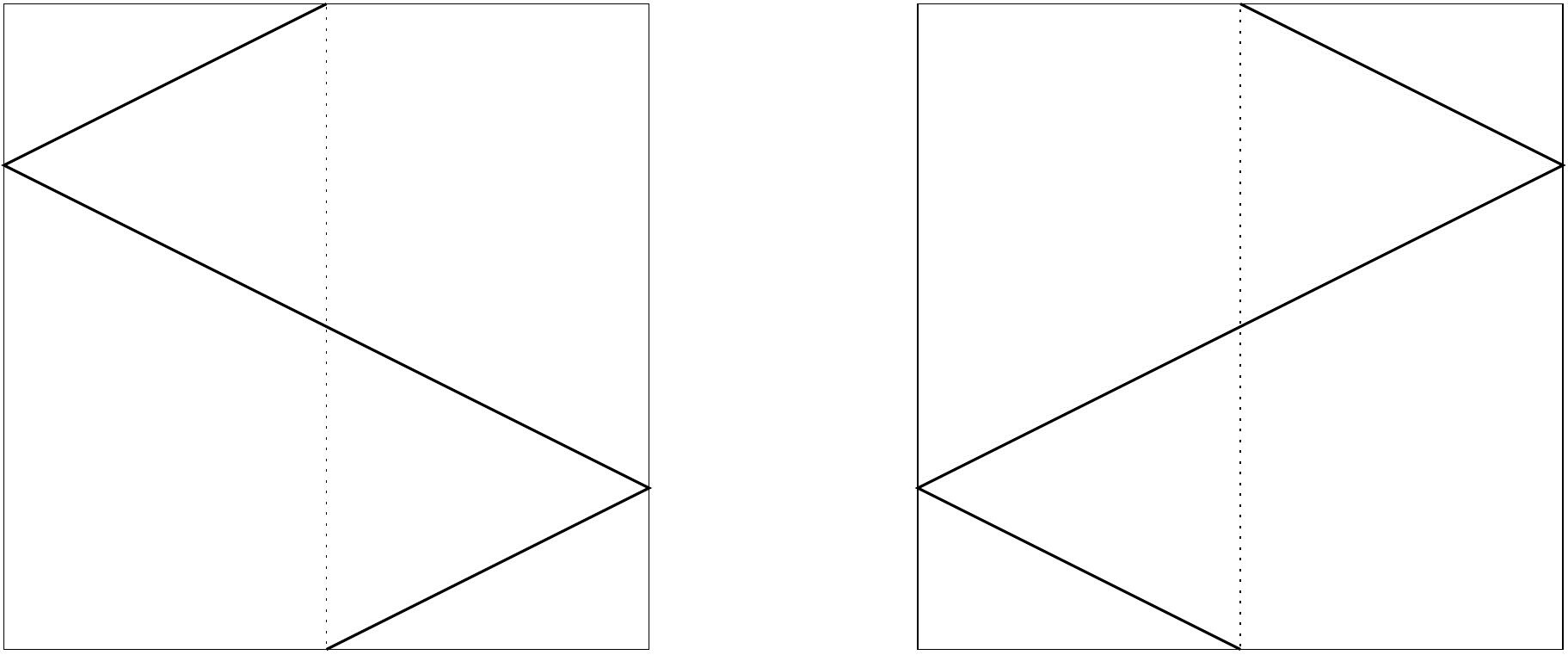}
\caption{\label{fig:grid_arc} Grids for arc permutations.}
\end{center}
\end{figure}

\subsection{Unimodal permutations}

In terms of grid classes, unimodal permutations are those that can be drawn on
one of the two pictures in Figure~\ref{fig:grid_arc_geo}, that is,
$$\U_n=\G_n\left(\begin{array}{c}
           1 \\
           -1 \\
         \end{array}
\right)\ \cup \ \G_n\left(\begin{array}{c}
           -1 \\
           1 \\
         \end{array}
\right).$$
\begin{figure}[htb]
\begin{center}
\includegraphics[width=3cm,angle=-90]{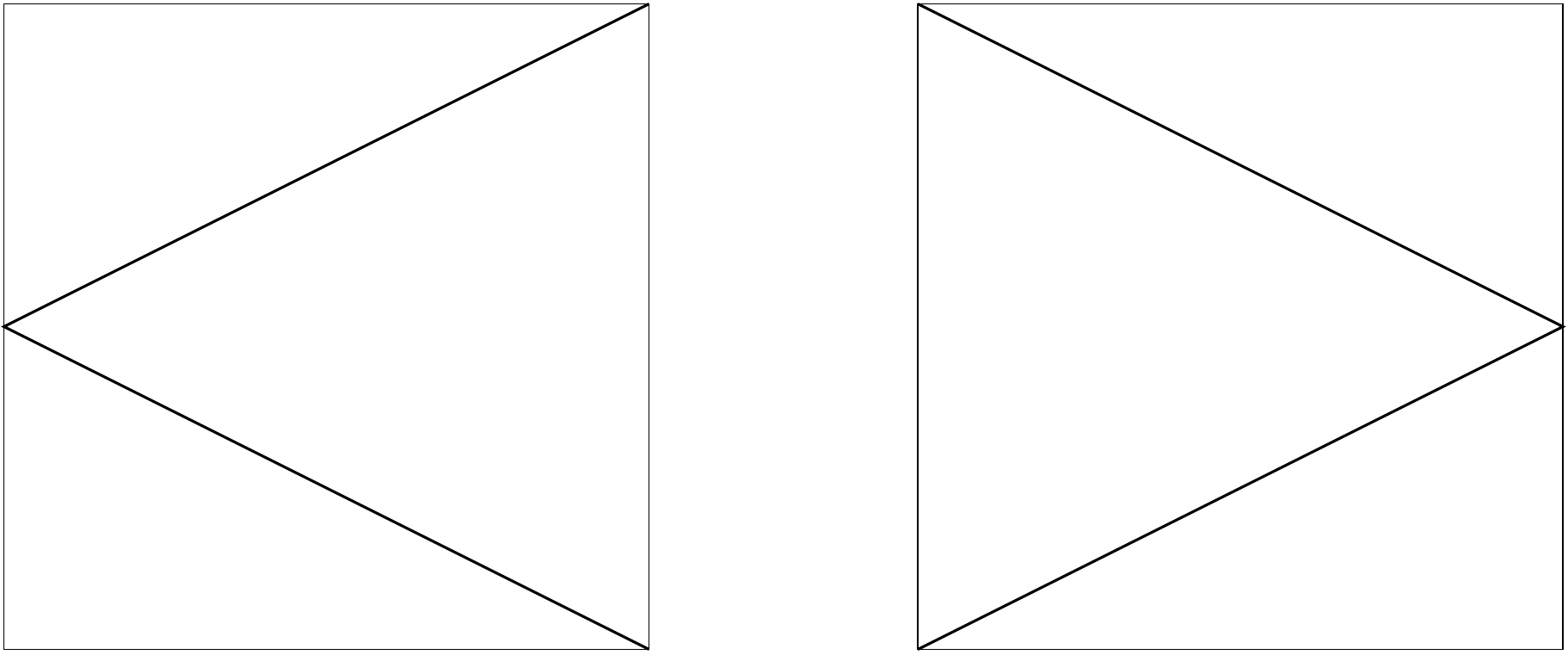}
\caption{\label{fig:grid_arc_geo} Grids for unimodal arc
permutations.}
\end{center}
\end{figure}

Next we characterize unimodal permutations in terms of pattern avoidance.

\begin{proposition}\label{unimodal-pattern}
$$\U_n=\A_n(2143,3412)
=\S_n(1324,1342,2143,2413,2431,3124,3142,3412,4213,4231).
$$
\end{proposition}

\begin{proof} If $\pi$ contains $2143$ or $3412$, then it is clear
that $\pi$ is not unimodal. For the converse, we show that
every arc permutation $\pi\in\A_n$ that is not unimodal must
contain one of the patterns $2143$ or $3412$. Since $\pi\in\A_n$,
it can be drawn on one of the two pictures in
Fig.~\ref{fig:grid_arc}. Suppose it can be drawn on the left
picture. Since $\pi$ is not unimodal, any drawing of $\pi$ on the
left picture requires some element $\pi(i)$ with $i>1$ to be on the
first increasing slope, and some
element $\pi(j)$ with $j<n$ to be on the second increasing slope.
Then $\pi(1)\pi(i)\pi(j)\pi(n)$ is an occurrence of $3412$. An
analogous argument shows that if $\pi$ can be drawn on the right
picture in Fig.~\ref{fig:grid_arc} but it is not unimodal, then it
contains $2143$.
\end{proof}

\begin{corollary}\label{unimodal-number} $|\S_n(1324,1342,2143,2413,2431,3124,3142,3412,4213,4231)|=2^n-2$ for $n\ge2$.
\end{corollary}

\section{Prefixes associated to the shifted staircase shape}\label{section-prefixes}

Consider the shifted staircase shape $\Delta_n$ with rows labeled
$1,2,\dots,n-1$ from top to bottom, and columns labeled
$2,3,\dots,n$ from left to right. Given a filling with the numbers
from $1$ to $n(n-1)/2$, with increasing entries in each row and
column, erase the numbers greater than $k$, for some $k$,
obtaining a {\em partial filling} of $\Delta_n$. For each of the
remaining entries $1\le r\le k$, if $r$ lies in row $i$ and
column $j$, let $t_r$ be the transposition $(i,j)$. Associate
to the partial filling the permutation $\pi=t_1t_2\dots t_k$, with multiplication from the right.

\begin{ex} The partial filling
$$\bt{rl} {\small\bt{c}\vspace{2.3mm}\\1\\ 2\\ 3\\ 4\\ 5\\ 6\et}\hspace*{-4mm}& \bt{c}\hspace{1.7mm}{\small2\hs 3\hs 4\hs 5\hs 6\hs 7} \vspace{1.5mm} \\
\young(12368\blank,:459\ten\blank,::7\blank\blank\blank,:::\blank\blank\blank,::::\blank\blank,:::::\blank)\et\et$$
corresponds to the product of transpositions
$$(1,2)(1,3)(1,4)(2,3)(2,4)(1,5)(3,4)(1,6)(2,5)(2,6)=4356217.$$
\end{ex}

\begin{theorem}\label{prefixes-theorem}
The set of permutations obtained as products of transpositions
associated to a partial filling of the shifted staircase shape
$\Delta_n$ is exactly $\LL_n$.
\end{theorem}

\begin{proof} The first observation is that if two boxes in the tableau
are in different rows and columns, the associated transpositions
commute. It follows that the resulting permutation depends only on
what boxes of the tableaux are filled, but not on the order in
which they were filled. For example, the partial filling
$$\bt{rl} {\small\bt{c}\vspace{2.3mm}\\1\\ 2\\ 3\\ 4\\ 5\\ 6\et}\hspace*{-4mm}& \bt{c}\hspace{1.7mm}{\small2\hs 3\hs 4\hs 5\hs 6\hs 7} \vspace{1.5mm} \\
\young(12345\blank,:6789\blank,::\ten\blank\blank\blank,:::\blank\blank\blank,::::\blank\blank,:::::\blank)\et\et$$
yields again the permutation $4356217$, just as the partial
filling in the above example, since both have the same set of
filled boxes.

We claim that, from the set of filled boxes, the corresponding
permutation can be read as follows. Let $i$ be the largest such
that the box $(i,i+1)$ is filled. Then, starting at the
bottom-left corner of that box, consider the path with north and
east steps (along the edges of the boxes of the tableau) that
separates the filled and unfilled boxes, ending at the top-right
corner. At each east step, read the label of the corresponding
column, and at each north step, read the label of the
corresponding row. This claim can be easily proved by induction on
the number of filled boxes. The permutations obtained by reading
the labels of  such paths are precisely the left-unimodal
permutations.
\end{proof}

The above proof gives a bijection between $\LL_n$ and the set of
shifted shapes of size at most $\binom{n}{2}$, which consist of the filled boxes in partial fillings.

\begin{definition}\label{def:shape} For $\pi\in \LL_n$, denote by $\shape(\pi)$ the shifted shape corresponding to any partial filling of $\Delta_n$ associated to $\pi$.
\end{definition}

\section{The weak order on $\U_n$}\label{section-order}

\subsection{A criterion for domination}

Let $\ell(\cdot)$ be the length function on the symmetric group
$\S_n$ with respect to the Coxeter generating set
$\s:=\{\sigma_i:\ 1\le i\le n-1\}$, where $\sigma_i$ is identified
with the adjacent transposition $(i,i+1)$.
Recall the
definition of the (right) weak order on
$\S_n$: for every pair $\pi,\tau\in \S_n$, $\pi\le \tau$ if
and only if $\ell(\pi)+\ell(\pi^{-1}\tau)=  \ell(\tau)$.
Denote this poset by $\W(\S_n)$. Recall that $\W(\S_n)$ is a
lattice, which is not modular. First, we give a combinatorial
criterion for weak domination of unimodal permutations.

The concept of shifted shape from Definition~\ref{def:shape} can
be extended to all unimodal permutations as follows: for
$\pi\in \U_n\setminus\LL_n$ let $\shape(\pi):=\shape(w_0\pi w_0)$,
where $w_0$ denotes the longest permutation $n\dots 21$, which is
the maximum in $\W(\U_n)$.  Denote by
$e=12\dots n$ the identity permutation, which is the minimum in $\W(\U_n)$.
Note that $w_0 \LL_n w_0=\U_n\setminus \LL_n\cup \{e, w_0\}$.

\begin{theorem}\label{weak_domination}
For every pair $\pi, \tau\in \U_n$, $\pi\le \tau$ in
$\W(\S_n)$ if and only if
\begin{itemize}
\item[(i)] either $\pi,\tau\in \LL_n$ or $\pi,\tau\in w_0 \LL_n w_0$, and
\item[(ii)] $\shape(\pi)\subseteq \shape(\tau)$.
\end{itemize}
\end{theorem}

\begin{proof}
By~\cite[Cor. 1.5.2, Prop. 3.1.3]{BB}, if $\pi\le \tau$ in
$\W(\S_n)$, then the corresponding descent sets satisfy
$\D(\pi^{-1})\subseteq \D(\tau^{-1})$. Combining this with the
characterizations of left-unimodal and unimodal permutations by
descent sets, given in Remarks~\ref{knuth1} and~\ref{knuth2},
condition $(i)$ follows.

Now we may assume, without loss of generality, that $\pi,
\tau\in \LL_n$ (for $\pi, \tau\in w_0 \LL_n w_0$, the same
proof holds by symmetry, by conjugation by $w_0$). To
complete the proof it suffices to show that for two left-unimodal
permutations, domination in weak order is equivalent to inclusion
of the corresponding shapes. Indeed, recall the bijection from $\LL_n$
to the set of shifted shapes of size at most $\binom{n}{2}$,
described in Section~\ref{section-prefixes}.
By this bijection, for any $\pi\in \LL_n\setminus\{w_0\}$, the addition
of a box in the border of $\shape(\pi)$ corresponds to a switch of two adjacent increasing letters in
$\pi$ giving a permutation in $\LL_n$. This is precisely the covering relation in
$\W(\U_n)$. Thus, for two left-unimodal permutations, the covering
relation in $\W(\U_n)$ is equivalent to the covering relation in the
poset of shifted shapes inside $\Delta_n$ ordered by inclusion,
and hence domination is equivalent.
\end{proof}

\begin{corollary}
For every $\pi\in \U_n$
$$
\ell(\pi)=|\shape(\pi)|,
$$
where $|\shape(\pi)|$ denotes the size of the shape.
\end{corollary}

\subsection{Enumeration of maximal chains}

Denote by $\W(\U_n)$ the subposet of $\W(\S_n)$ which is induced by
$\U_n$. Theorem~\ref{weak_domination} implies the following nice properties of this poset.

\begin{corollary}\label{weak-properties}
$\W(\U_n)$ is a graded self-dual modular lattice.
\end{corollary}

\begin{corollary}\label{mc_enumeration_1}
For every $\pi\in\U_n\setminus \{w_0\}$, the number of maximal
chains in the interval $[e,\ \pi]$ is equal to the number of
standard Young tableaux of shifted shape $\shape(\pi)$, hence given by a
hook formula.
\end{corollary}

\begin{proof}
By Theorem~\ref{prefixes-theorem} together with
Theorem~\ref{weak_domination}, the statement holds for every
$\pi\in \LL_n\setminus\{w_0\}$. By conjugation by $w_0$, it holds
for all elements in $\U_n\setminus \LL_n$ as well.
\end{proof}

\begin{corollary}\label{mc_enumeration}
For $n>2$, the number of maximal chains in $\W(\U_n)$ is equal to
twice the number of standard Young tableaux of shifted staircase
shape, hence equal to
$$
2\binom{n}{2}! \cdot \prod_{i=0}^{n-2} \frac{i!}{(2i+1)!}.
$$
\end{corollary}

\begin{proof}
The maximum $w_0$ covers the two elements $w_0\sigma_1$ and $w_0\sigma_{n-1}$.
Thus the number of maximal chains in $\W(\U_n)$ is the sum of the numbers of maximal
chains in $[e,w_0\sigma_1]$ and $[e,w_0\sigma_{n-1}]$. By
Corollary~\ref{mc_enumeration_1}, this equals the number of
standard Young tableaux of shape $\shape(w_0\sigma_1)$ plus number of
standard Young tableaux of shape $\shape(w_0\sigma_{n-1})$. Since
$w_0(w_0\sigma_{n-1})w_0= w_0\sigma_1$, these two shapes are the same, namely
$\Delta_n$ with the box in row $n-1$ (the bottommost row) removed. By Schur's Formula~\cite{Schur}\cite[p.
267 (2)]{Md}, the number of standard Young tableaux of this shape
is $\binom{n}{2}! \cdot \prod_{i=0}^{n-2} \frac{i!}{(2i+1)!}$, completing the proof.
\end{proof}

\subsection{The Hasse diagram}

Let $\Gamma_n$ be the undirected Hasse diagram of $\W(\U_n)$. A drawing of $\Gamma_4$ is given by the black vertices and solid edges in Figure~\ref{fig:X4}.

\begin{proposition}\label{U-Hasse}
\begin{itemize}
\item[(i)] The diameter of $\Gamma_n$ is $\binom{n}{2}$.
\item[(ii)] The vertices $e$ and $w_0$ are antipodes in
$\Gamma_n$. \item[(iii)] The number of geodesics between $e$ and $w_0$ is $2\binom{n}{2}! \cdot \prod_{i=0}^{n-2}
\frac{i!}{(2i+1)!}$.
\end{itemize}
\end{proposition}

\begin{proof} Since $\W(\U_n)$ is a modular lattice, the distance
between any two vertices is equal to the difference between the
ranks of their join and their meet (see~\cite[Lemma 5.2]{TFT}).
Hence, the diameter is equal to the maximum rank.
This proves $(i)$ and $(ii)$.
Part $(iii)$ then follows from Corollary~\ref{mc_enumeration}.
\end{proof}

\section{A graph structure on arc permutations}\label{X-graph}

\subsection{The graph $X_n$}

Let $X_n$ be the subgraph of the Cayley graph $X(\S_n,\s)$ induced
by $\A_n$. In other words, the vertex set of $X_n$ is $\A_n$, and
two elements $u,v\in \A_n$ are adjacent if and only if there
exists a simple reflection $\sigma_i\in\s$, such that
$u=v\sigma_i$. The graph $X_4$ is drawn in Figure~\ref{fig:X4}. The following theorem shows that $X_n$ and
$\Gamma_n$ share similar properties.

\begin{theorem}\label{A-graph}
\begin{itemize}
\item[(i)] The diameter of $X_n$ is $\binom{n}{2}$.
\item[(ii)] The vertices $e$ and $w_0$ are antipodes in $X_n$.
\item[(iii)] The number of vertices in geodesics between $e$ and $w_0$ is $2^n-2$.
\item[(iv)] The number of geodesics between $e$ and $w_0$ is $2\binom{n}{2}! \cdot \prod_{i=0}^{n-2} \frac{i!}{(2i+1)!}$.
\end{itemize}
\end{theorem}

This theorem will be proved Subsections~\ref{X-diameter} and~\ref{sec:Agraph}.

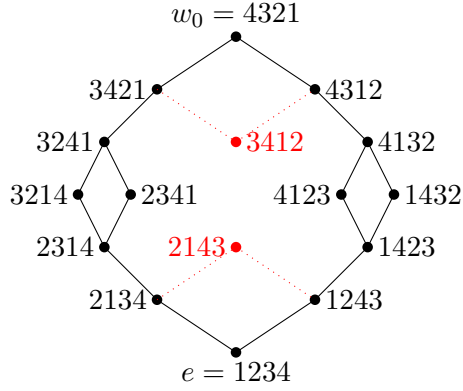
\begin{figure}[hbt]
\begin{center}

\begin{tikzpicture}[scale=0.7]
\fill (3,6) circle (0.1) node[above]{$w_0=4321$};
\fill (1.5,5) circle (0.1) node[left]{$3421$};
\fill (0.5,4) circle (0.1) node[left]{$3241$};
\fill (0,3) circle (0.1) node[left]{$3214$};
\fill (1,3) circle (0.1) node[right]{$2341$};
\fill (0.5,2) circle (0.1) node[left]{$2314$};
\fill (1.5,1) circle (0.1) node[left]{$2134$};
\fill (3,0) circle (0.1) node[below]{$e=1234$};

\fill (4.5,5) circle (0.1) node[right]{$4312$};
\fill (5.5,4) circle (0.1) node[right]{$4132$};
\fill (6,3) circle (0.1) node[right]{$1432$};
\fill (5,3) circle (0.1) node[left]{$4123$};
\fill (5.5,2) circle (0.1) node[right]{$1423$};
\fill (4.5,1) circle (0.1) node[right]{$1243$};

\fill[red] (3,4) circle (0.1) node[right]{$3412$};
\fill[red] (3,2) circle (0.1) node[left]{$2143$};

\draw (3,0)--(1.5,1)--(0.5,2)--(0,3)--(0.5,4)--(1.5,5)--(3,6);
\draw (3,0)--(4.5,1)--(5.5,2)--(6,3)--(5.5,4)--(4.5,5)--(3,6);
\draw[red,dotted] (1.5,1)--(3,2)--(4.5,1);
\draw (0.5,2)--(1,3)--(0.5,4);
\draw (5.5,2)--(5,3)--(5.5,4);
\draw[red,dotted] (1.5,5)--(3,4)--(4.5,5);
\end{tikzpicture}
\caption{\label{fig:X4} The graph $X_4$. The vertices not lying in a geodesic between $e$ and $w_0$ are drawn in red with dotted edges, and they correspond to non-unimodal permutations by Lemma~\ref{lemma-unimodal}.}
\end{center}
\end{figure}

\subsection{The diameter of $X_n$}\label{X-diameter}

In this subsection we show that the diameter of $X_n$ is $\binom{n}{2}$, proving Theorem~\ref{A-graph}(i).
To see that this is a lower bound, note that the inversion number does not change by more than 1 along
each edge of $X_n$. It follows that the diameter of $X_n$ is at least $\inv(w_0)-\inv(e)=\binom{n}{2}$.
This argument also shows that part (ii) of Theorem~\ref{A-graph} will follow once we prove part $(i)$, since the distance between $e$ and $w_0$ is at least $\binom{n}{2}$.

\ms

The proof that this is also an upper bound on the diameter is more involved, and it is similar to
the proof in~\cite[Theorem 5.1]{TFT}.
Consider the encoding $\psi:\A_n \rightarrow \{0,1,\dots,n-1\}\times \{0,1\}^{n-2}$ given by
$\psi(\pi)=(\psi_0,\psi_1,\dots,\psi_{n-2})$, where
$$
\psi_0:=\pi(1)-1
$$
and, for $1\le i \le n-2$,
$$
\psi_i:=\begin{cases} 1 & \text{if } \ol{\pi(i+1)-1}\in\{\pi(1),\pi(2),\dots,\pi(i)\},\\
0 &  \text{if } \ol{\pi(i+1)+1}\in\{\pi(1),\pi(2),\dots,\pi(i)\},
\end{cases}
$$
where $\ol{m}$ denotes the element of $\{1,2,\dots,n\}$ that is congruent with $m$ mod $n$.
Note that exactly one of the two above conditions holds, because $\{\pi(1),\pi(2),\dots,\pi(i)\}$ forms an interval in $\ZZ_n$.

\begin{ex} For $\pi=4352176\in \S_7$, $\psi(\pi)=(3,0,1,0,0,0)$.
\end{ex}

The encoding of the vertices of $X_4$ is given in Figure~\ref{fig:X4-encoded}. The following observation is clear from the definition of $X_n$ and the encoding $\psi$.

\begin{lemma}\label{adjacency-encoding}
Two arc permutations $\pi,\tau\in\A_n$ with $\pi\neq\tau$ are adjacent in $X_n$ if
and only if exactly one of the following holds:
\begin{itemize}
\item[(i)] $\psi(\tau)$ is obtained from $\psi(\pi)$ by switching
two adjacent entries $\psi_i$ and $\psi_{i+1}$ for some $1\le i<n-2$;

\item[(ii)]
$\psi(\pi)_i=\psi(\tau)_i$ for all $0\le i<n-2$;

\item[(iii)]
$\psi(\pi)_{0}+\psi(\pi)_1=\psi(\tau)_{0}+\psi(\tau)_1$ mod $n$, and $\psi(\tau)_i=\psi(\pi)_i$ for all $2\le i\le n-2$.
\end{itemize}
\end{lemma}

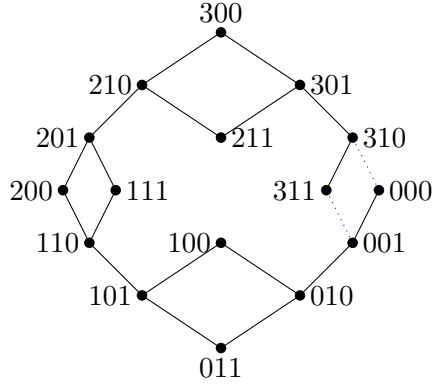
\begin{figure}[htb]
\begin{center}

\begin{tikzpicture}[scale=0.7]
\fill (3,6) circle (0.1) node[above]{$300$};
\fill (1.5,5) circle (0.1) node[left]{$210$};
\fill (0.5,4) circle (0.1) node[left]{$201$};
\fill (0,3) circle (0.1) node[left]{$200$};
\fill (1,3) circle (0.1) node[right]{$111$};
\fill (0.5,2) circle (0.1) node[left]{$110$};
\fill (1.5,1) circle (0.1) node[left]{$101$};
\fill (3,0) circle (0.1) node[below]{$011$};
\fill (4.5,5) circle (0.1) node[right]{$301$};
\fill (5.5,4) circle (0.1) node[right]{$310$};
\fill (6,3) circle (0.1) node[right]{$000$};
\fill (5,3) circle (0.1) node[left]{$311$};
\fill (5.5,2) circle (0.1) node[right]{$001$};
\fill (4.5,1) circle (0.1) node[right]{$010$};
\fill (3,4) circle (0.1) node[right]{$211$};
\fill (3,2) circle (0.1) node[left]{$100$};

\draw (3,0)--(1.5,1)--(0.5,2)--(0,3)--(0.5,4)--(1.5,5)--(3,6);
\draw (3,0)--(4.5,1)--(5.5,2)--(6,3); \draw
(5.5,4)--(4.5,5)--(3,6); \draw[blue,dotted] (6,3)--(5.5,4); \draw
(1.5,1)--(3,2)--(4.5,1); \draw (0.5,2)--(1,3)--(0.5,4);
\draw[blue,dotted] (5.5,2)--(5,3); \draw (5,3)--(5.5,4); \draw
(1.5,5)--(3,4)--(4.5,5);
\end{tikzpicture}
\caption{\label{fig:X4-encoded} The graph $X_4$ with its vertices encoded by $\psi$ (commas and parenthesis have been removed).
Deleting the two dotted blue edges gives the undirected Hasse diagram of
the dominance order on $\{0,1,2,3\}\times\{0,1\}^2\subset\ZZ^3$.}
\end{center}
\end{figure}

The set
$\{0,1,\dots,n-1\}\times \{0,1\}^{n-2}$
of possible encodings inherits the dominance order from $\ZZ^{n-1}$, that is, $v\le u$ if and
only if for every $1\le i\le n-1$, $$\sum\limits_{j=1}^i v_j\le \sum\limits_{j=1}^i u_j.$$
The covering relations in this poset are almost identical to those described by Lemma~\ref{adjacency-encoding}. More precisely, we have the following result.

\begin{proposition}\label{iso}
Through the encoding $\psi$, the graph $X_n$ is isomorphic to the undirected Hasse diagram of the
dominance order on $\{0,1,\dots,n-1\}\times \{0,1\}^{n-2}$ with the $2^{n-3}$ additional edges arising from
Lemma~\ref{adjacency-encoding}(iii) with $\{\psi(\tau)_{0},\psi(\pi)_{0}\}=\{n-1,0\}$.
\end{proposition}

Denote by $\ddom$ the distance function in the undirected Hasse diagram of the dominance order.
To compute $\ddom(\psi(\pi),\psi(\tau))$, let us first recall some
basic facts. The dominance order on $\ZZ^{n-1}$
is a ranked poset where
$$
{\rm rank} (v_1,\dots,v_{n-1})= \sum\limits_{j=1}^{n-1}
(n-j)v_j=\sum\limits_{j=1}^{n-1} \sum\limits_{k=1}^j v_k.
$$
This poset is a modular lattice, with
$$
(v_1,\dots,v_{n-1})\wedge
(u_1,\dots,u_{n-1})=(\alpha_1,\dots,\alpha_{n-1})
$$
where $\alpha_k=\min\{\sum\limits_{i=1}^k v_i,
\sum\limits_{i=1}^k u_i\}-\min\{\sum\limits_{i=1}^{k-1} v_i,
\sum\limits_{i=1}^{k-1} u_i\}$ for every $1\le k\le n-1$, and
$$
(v_1,\dots,v_{n-1})\vee
(u_1,\dots,u_{n-1})=(\beta_1,\dots,\beta_{n-1})
$$
where $\beta_k=\max\{\sum\limits_{i=1}^k v_i, \sum\limits_{i=1}^k
u_i\}-\max\{\sum\limits_{i=1}^{k-1} v_i, \sum\limits_{i=1}^{k-1}
u_i\}$ for every $1\le k\le n-1$.
 Finally, recall that the distance between two elements in the
undirected Hasse diagram of a modular lattice is equal to the
difference between the ranks of their join and their meet, see
e.g.~\cite[Lemma 5.2]{TFT}.

Combining these facts implies that
\begin{multline}
\ddom(\psi(\pi),\psi(\tau))={\rm rank}(\psi(\pi)\vee \psi(\tau))-{\rm rank}(\psi(\pi)\wedge
\psi(\tau))\\
= \sum\limits_{j=0}^{n-2} \sum\limits_{k=0}^j
(\max\{\sum\limits_{i=0}^k \psi(\pi)_i, \sum\limits_{i=0}^k
\psi(\tau)_i\}-\max\{\sum\limits_{i=0}^{k-1} \psi(\pi)_i,
\sum\limits_{i=0}^{k-1} \psi(\tau)_i\})
\\
-\sum\limits_{j=0}^{n-2} \sum\limits_{k=0}^j
(\min\{\sum\limits_{i=0}^k \psi(\pi)_i, \sum\limits_{i=0}^k
\psi(\tau)_i\}-\min\{\sum\limits_{i=0}^{k-1} \psi(\pi)_i,
\sum\limits_{i=0}^{k-1} \psi(\tau)_i\})\\
\label{eq1}
= \sum\limits_{j=0}^{n-2}|\sum\limits_{i=0}^j
(\psi(\pi)_i-\psi(\tau)_i)|.
\end{multline}

Now we are ready to prove the upper bound
on diameter of $X_n$. Denoting by $\dXn$ the distance function in $X_n$,
we will show that for any $\pi,\tau\in\A_n$, $\dXn(\pi,\tau)\le\binom{n}{2}$.

Let $\gamma$ be the $n$-cycle $(1,\dots,n)$. Clearly, $\A_n$ is
invariant under left multiplication by $\gamma$. Moreover, left
multiplication by $\gamma$ is an automorphism of $X_n$. Thus, for
any integer $k$,
\beq\label{eq:dpitau}\dXn(\pi,\tau)=\dXn(\gamma^k \pi,\gamma^k \tau)\le \ddom(\psi(\gamma^k \pi),\psi(\gamma^k \tau)),\eeq
where the last inequality follows from Proposition~\ref{iso}.
Let $x_0=\psi(\pi)_0-\psi(\tau)_0$.
By equation~\eqref{eq1},
\beq\label{eq:dpitau0}\ddom(\psi(\pi),\psi(\tau))=\sum_{j=0}^{n-2}|x_0+\sum_{i=1}^j
(\psi(\pi)_i-\psi(\tau)_i)|=\sum_{j=0}^{n-2}|x_j|,\eeq
where $x_j=x_0+\sum_{i=1}^j (\psi(\pi)_i-\psi(\tau)_i)$ for $1\le j\le n-2$. Note that $|x_{j}-x_{j-1}|\le 1$ for every $j$.
If $x_0=0$, then
$$\dXn(\pi,\tau)\le\ddom(\psi(\pi),\psi(\tau))=\sum_{j=0}^{n-2}|x_j|\le 0+1+\dots+(n-2)=\binom{n-1}{2}$$ and we are done.

Otherwise, we can assume without loss of generality that $1\le x_0\le n-1$. Let $k=-\tau(1)$, so that $\psi(\gamma^k\pi)_0-\psi(\gamma^k\tau)_0=x_0-n$.
Note that for $1\le i\le n-2$, we have $\psi(\gamma^k \pi)_i=\psi(\pi)_i$ and $\psi(\gamma^k \tau)_i=\psi(\tau)_i$. Thus, by equation~\eqref{eq1},
$$\ddom(\psi(\gamma^k \pi),\psi(\gamma^k \tau))=\sum_{j=0}^{n-2}|x_0-n+\sum_{i=1}^j
(\psi(\pi)_i-\psi(\tau)_i)|=\sum_{j=0}^{n-2}|x_j-n|.$$
Combining this formula with equations~\eqref{eq:dpitau} and~\eqref{eq:dpitau0}, we get
$$\dXn(\pi,\tau)\le \min\{\sum_{j=0}^{n-2}|x_j|,\sum_{j=0}^{n-2}|n-x_j|\}.$$
If $0\le x_j\le n$ for all $j$, then
$$\dXn(\pi,\tau)\le \frac{1}{2}\left(\sum_{j=0}^{n-2}|x_j|+\sum_{j=0}^{n-2}|n-x_j|\right)=\frac{1}{2}\sum_{j=0}^{n-2}(x_j+n-x_j)=\binom{n}{2}.$$
Otherwise, since $|x_{j}-x_{j-1}|\le 1$ for all $j$, there must be some $i$ such that $x_i=0$ or $x_i=n$. If $x_i=0$ for some $i$, then $|x_j|\le|j-i|$ for all $j$, so $\sum_{j=0}^{n-2}|x_j|\le\binom{n}{2}$.
Similarly, if $x_j=n$ for some $j$, then $\sum_{j=0}^{n-2}|n-x_j|\le\binom{n}{2}$, completing the proof of Theorem~\ref{A-graph}(i).

\subsection{Geodesics of $X_n$}\label{sec:Agraph}

To prove parts (ii) and (iii) of Theorem~\ref{A-graph} we need the
following lemma.

\begin{lemma}\label{lemma-unimodal}
A permutation in $\A_n$ lies in a geodesic between $e$ and $w_0$ if and only if it is unimodal.
\end{lemma}

\begin{proof}
By Corollary~\ref{weak-properties}, all unimodal permutations lie in
geodesics between $e$ and $w_0$ in the undirected Hasse diagram of $\W(\U_n)$. By
Proposition~\ref{U-Hasse} and Theorem~\ref{A-graph}(i), the
distance between $e$ and $w_0$ in this Hasse diagram is the same
as in $X_n$, thus the geodesics between these vertices in this
Hasse diagram are also geodesics in $X_n$.

It remains to show that for every non-unimodal arc permutations $\pi\in\Z_n$,
$\pi$ is not in a geodesic between $e$ and $w_0$.
It suffices to prove that for every such $\pi$, either $\dXn(e,\pi)>\ell(\pi)$,
or $\dXn(w_0,\pi)>\binom{n}{2}-\ell(\pi)$. These two cases are analogous to
the dichotomy in Remark~\ref{knuth3} and Figure~\ref{fig:grid_arc}.

If $\pi^{-1}(1)>\pi^{-1}(n)$, then $\pi^{-1}(n-1)<\pi^{-1}(n)<\pi^{-1}(1)<\pi^{-1}(2)$, since otherwise $\pi$ would be unimodal.
Let $\ell=\ell(\pi)$, and suppose for contradiction that $\dXn(e,\pi)=\ell$.
Then there is a sequence of arc permutations $\pi=\pi_\ell,\pi_{\ell-1},\dots,\pi_1=e$ where each $\pi_i$ is obtained from $\pi_{i+1}$
by switching two adjacent letters at a descent, decreasing the number of inversions by one. In particular, in every $\pi_i$, the entry $n-1$ is to the left of $n$, and $2$ is to the right of $1$.
In order to remove the inversion created by the pair $(1,n)$ in $\pi$ we would have to switch $1$ and $n$, which would create a permutation containing $3142$, thus not in $\A_n$ by Theorem~\ref{arc-pattern-thm}.
This shows that $\dXn(e,\pi)>\ell(\pi)$.

Similarly, if $\pi^{-1}(1)<\pi^{-1}(n)$, then $\pi^{-1}(2)<\pi^{-1}(1)<\pi^{-1}(n)<\pi^{-1}(n-1)$.
Let $k=\binom{n}{2}-\ell(\pi)$, and suppose for contradiction that $\dXn(w_0,\pi)=k$.
Then there is a sequence of arc permutations
$\pi=\pi_k, \pi_{k-1},\dots,\pi_1=w_0$
where each $\pi_i$ is obtained from $\pi_{i+1}$
by switching two adjacent letters at an ascent, increasing the number of inversions by one.
Again, this is impossible because after switching the pair $(1,n)$, the entries $2n1(n-1)$ would form an occurrence of $2413$, so the permutation would not be in $\A_n$ by Theorem~\ref{arc-pattern-thm}.
We conclude that $\dXn(w_0,\pi)>\binom{n}{2}-\ell(\pi)$.
\end{proof}

\begin{proof}[Proof of Theorem~\ref{A-graph}] Parts $(i)$ and $(ii)$ were proved in
Subsection~\ref{X-diameter}. To prove (iii), combine Lemma~\ref{lemma-unimodal} with
Claim~\ref{claimU}.
Finally, (iv) follows from
Lemma~\ref{lemma-unimodal} together with
Corollary~\ref{mc_enumeration}.
\end{proof}

\section{Equidistribution}\label{equid}

In this section we show that the descent set is equidistributed on arc permutations that are not unimodal and on the set of standard Young tableaux
obtained from hooks by adding one box in position $(2,2)$.

\subsection{Enumeration of arc permutations by descent set}

For a set $D=\{i_1,\dots,i_k\}$, define $\x^{D}=x_{i_1}\dots
x_{i_k}$.

\begin{proposition}\label{prop:DUn} For $n\ge2$,
$$\sum_{\pi\in\A_n}\x^{\D(\pi)}=(1+x_1)\dots(1+x_{n-1})\left(1+\sum_{i=1}^{n-2}\frac{x_i+x_{i+1}}{(1+x_i)(1+x_{i+1})}\right).$$
\end{proposition}

\begin{proof}
Let $\pi\in\A_n$, and let $i=\max\{\pi^{-1}(1),\pi^{-1}(n)\}-1$.

If $i=n-1$, then $\pi$ can be drawn on the picture on the left of
Figure~\ref{fig:grid_arc_geo}.
The generating function for these permutations with respect to the
descent set is $(1+x_1)\dots(1+x_{n-1})$. Indeed, each $\pi(j)$ for
$2\le j\le n$ is either larger or smaller than all the previous
entries, and causes a descent with $\pi(j-1)$ only in the second
case (this is when $\pi(j)$ corresponds to a dot on the descending
slope in the picture). So, $\pi(j)$ contributes a factor
$1+x_{j-1}$ to the generating function.

Let us now consider permutations with fixed $i$, with $1\le i\le
n-2$. Since $\pi(1)\dots\pi(i)$ can be drawn on the picture on the
left of Figure~\ref{fig:grid_arc_geo}, the same reasoning as above
shows that the contribution of the descents of $\pi(1)\dots\pi(i)$
to the generating function is $(1+x_1)\dots(1+x_{i-1})$. Now we
have $\pi(i+1)\in\{1,n\}$, by the choice of $i$. If $\pi(i+1)=1$
(resp. $\pi(i+1)=n$), then we can draw $\pi$ on the picture on the
left (resp. right) of Figure~\ref{fig:grid_arc}, with $\pi(i+1)$
being the first entry to the right of the vertical dotted line. In
this case, the descent $\pi(i)\pi(i+1)$ contributes $x_{i}$
(resp. the descent $\pi(i+1)\pi(i+2)$ contributes $x_{i+1}$) to
the generating function. In both cases, each one of the entries
$\pi(j)$ with $i+2\le j\le n-1$ will produce a descent
$\pi(j)\pi(j+1)$ iff the corresponding dot is on the descending
slope to the right of the dotted line. Thus, $\pi(j)$ contributes a
factor $1+x_{j}$ for each $i+2\le j\le n-1$.

Combining all these contributions we get the generating function
\begin{eqnarray*}\sum_{\pi\in\A_n}\x^{\D(\pi)}&=&(1+x_1)\dots(1+x_{n-1})+\sum_{i=1}^{n-2}(1+x_1)\dots(1+x_{i-1})(x_i+x_{i+1})(1+x_{i+2})\dots(1+x_{n-1})\\
&=&(1+x_1)\dots(1+x_{n-1})\left(1+\sum_{i=1}^{n-2}\frac{x_i+x_{i+1}}{(1+x_i)(1+x_{i+1})}\right).
\end{eqnarray*}
\end{proof}

\begin{corollary}\label{prop:majUn} Let $\maj(\pi)$ denote the major index of $\pi$. For $n\ge2$,
$$\sum_{\pi\in\A_n}q^{\maj(\pi)}=(1+q)\dots(1+q^{n-2})[n]_q.$$
\end{corollary}

\begin{proof}
The generating function for the major index is obtained by
replacing $x_i$ with $q^i$ for each $1\le i\le n-1$ in the formula
from Proposition~\ref{prop:DUn}:
$$\sum_{\pi\in\A_n}q^{\maj(\pi)}=(1+q)\dots(1+q^{n-1})\left(1+(1+q)\sum_{i=1}^{n-2}\frac{q^i}{(1+q^i)(1+q^{i+1})}\right).$$
The summation inside the parentheses can be simplified as
\begin{eqnarray*}\sum_{i=1}^{n-2}\frac{q^i}{(1+q^i)(1+q^{i+1})}&=&
\frac{1}{q-1}\sum_{i=1}^{n-2}\left(\frac{1}{1+q^i}-\frac{1}{1+q^{i+1}}\right)\\
&=&\frac{1}{q-1}\left(\frac{1}{1+q}-\frac{1}{1+q^{n-1}}\right)=\frac{q-q^{n-1}}{(1-q)(1+q)(1+q^{n-1})}.
\end{eqnarray*}
Putting it back in the original equation,
\begin{eqnarray*}\sum_{\pi\in\A_n}q^{\maj(\pi)}&=&
(1+q)\dots(1+q^{n-1})\left(1+\frac{q-q^{n-1}}{(1-q)(1+q^{n-1})}\right)\\
&=&(1+q)\dots(1+q^{n-2})\left(1+q^{n-1}+q\frac{1-q^{n-2}}{1-q}\right)=(1+q)\dots(1+q^{n-2})[n]_q.
\end{eqnarray*}
\end{proof}

\subsection{Non-unimodal arc permutations}

Recall that $\Z_n$ denotes the set of arc permutations that are not unimodal. Let $f^\lambda$ denote the number of standard Young tableaux of shape $\lambda$.

\begin{proposition}\label{non-unimodal-enumeration}
For $n\ge4$,
$$
|\Z_n|=\sum\limits_{k=2}^{n-2} f^{(k,2,1^{n-k-2})}.
$$
\end{proposition}

\begin{proof}
By Claims~\ref{claimU}
and~\ref{arc-enumeration},
it is clear that
$$|\Z_n|=n2^{n-2}-2^n+2=2^{n-2}(n-4)+2.$$
On the other hand, using the hook-length formula, we obtain
$$\sum\limits_{k=2}^{n-2}
f^{(k,2,1^{n-k-2})}=\sum\limits_{k=2}^{n-2}\frac{(k-1)(n-k-1)}{n-1}\binom{n}{k}=2^{n-2}(n-4)+2,$$
where the last step follows from easy manipulations of binomial coefficients.
\end{proof}

\begin{proposition}For $n\ge2$,
$$\sum_{\pi\in\Z_n}\x^{\D(\pi)}=\sum_{\pi\in\A_n}\x^{\D(\pi)}-2(1+x_1)\dots(1+x_{n-1})+1+x_1\dots x_{n-1}.$$
\end{proposition}

\begin{proof}
Since $\Z_n=\A_n\setminus\U_n$, the statement to be proved is equivalent to
$$\sum_{\pi\in\U_n}\x^{\D(\pi)}=2(1+x_1)\dots(1+x_{n-1})-1-x_1\dots x_{n-1}.$$
Unimodal arc permutations are those that can be drawn on one of
the pictures in Figure~\ref{fig:grid_arc_geo}. We have shown in the
proof of Proposition~\ref{prop:DUn} that for permutations that can
be drawn on the left picture, the generating function for the
descent set is $(1+x_1)\dots(1+x_{n-1})$. We obtain the same
generating function for permutations that can be drawn on the
right picture, since each $\pi(j)$ for $1\le j\le n-1$ causes a
descent with $\pi(j+1)$ iff it is drawn on the descending slope of
the grid, thus contributing a factor $1+x_{j}$. Finally, the we
have to subtract the contribution of the only two permutations
that can be drawn on both grids, which are $12\dots n$ and
$n\dots21$.
\end{proof}

\begin{corollary}For $n\ge2$,
$$\sum_{\pi\in\Z_n}q^{\maj(\pi)}=(1+q)\dots(1+q^{n-2})[n]_q-2(1+q)\dots(1+q^{n-1})+1+q^{\binom{n}{2}}.$$
\end{corollary}

\subsection{Standard Young tableaux of shape $(k,2,1^{n-k-2})$}

Let $\Hk_n$ be the set of standard Young tableaux of shape
$(k,1^{n-k})$ (a hook) for some $1\le k\le n$. Let $\T_n$ be the
set of standard Young tableaux of shape $(k,2,1^{n-k-2})$ for some
$2\le k\le n-2$.

\begin{lemma}
$$\sum_{T\in\Hk_n}\x^{\D(T)}=(1+x_1)\dots(1+x_{n-1}).$$
\end{lemma}

\begin{proof}
Any $T\in\Hk_n$ has a $1$ in the upper-left corner. $T$ is now
determined by the set of entries $j$ with $2\le j\le n$ that are in
the first column, since the rest have to be in the first row. Each such
$j$ creates a descent with $j-1$ iff it is in the first column,
which gives the contribution $1+x_{j-1}$ to the generating
function.
\end{proof}

\begin{theorem}\label{thm:TZ}
For $n\ge4$,
$$\sum_{T\in\T_n}\x^{\D(T)}=\sum_{\pi\in\Z_n}\x^{\D(\pi)}.$$
\end{theorem}

\begin{proof}
Given $T\in\T_n$, let $i+2$ be the element in the box in the second
row and second column. Note that $2\le i\le n-2$. There are two
possibilities for $i+1$: it is either in the first row or in the
first column. \bit \item If $i+1$ is in the first row, the entries
$1,2,\dots,i$ form an arbitrary hook with more than one row. As in
the above lemma, the descents of these entries contribute
$(1+x_1)\dots(1+x_{i-1})-1$ to the generating function, with the
$-1$ corresponding to the invalid one-row tableau. Now $i$ is not
a descent of $T$ but $i+1$ is, producing a factor $x_{i+1}$. Each
one of the remaining entries $j$ with $i+3\le j\le n$ can be in
the first row or in the first column, and it creates a descent
with $j-1$ iff it is in the first column, which gives the
contribution $1+x_{j-1}$ to the generating function. Thus, this
case gives a summand
$$[(1+x_1)\dots(1+x_{i-1})-1]x_{i+1}(1+x_{i+2})\dots(1+x_{n-1}).$$
\item If $i+1$ is in the first column, the entries $1,2,\dots,i$
form an arbitrary hook with more than one column. The descents of
these entries contribute $(1+x_1)\dots(1+x_{i-1})-x_1\dots
x_{i-1}$ to the generating function, subtracting the invalid
one-column tableau. Now $i$ is a descent of $T$ but $i+1$ is not,
producing a factor $x_{i}$. As in the previous case, each one of
the remaining entries $j$ with $i+3\le j\le n$ contributes a
factor $1+x_{j-1}$ to the generating function. Thus, this case
gives a summand $$[(1+x_1)\dots(1+x_{i-1})-x_1\dots
x_{i-1}]x_{i}(1+x_{i+2})\dots(1+x_{n-1}).$$ \eit We have proved
that
\begin{eqnarray}\nn\sum_{T\in\T_n}\x^{\D(T)}&=&\sum_{i=2}^{n-2}[(1+x_1)\dots(1+x_{i-1})-1]x_{i+1}(1+x_{i+2})\dots(1+x_{n-1})\\
\nn&&+\sum_{i=2}^{n-2}[(1+x_1)\dots(1+x_{i-1})-x_1\dots x_{i-1}]x_{i}(1+x_{i+2})\dots(1+x_{n-1})\\
\nn&=&\sum_{i=1}^{n-2}(1+x_1)\dots(1+x_{i-1})(x_i+x_{i+1})(1+x_{i+2})\dots(1+x_{n-1})\\
\label{eq:Tnsum}&&-\sum_{i=1}^{n-2}(x_{i+1}+x_1\dots
x_{i})(1+x_{i+2})\dots(1+x_{n-1}).
\end{eqnarray}
The last sum above can be simplified using the following two telescopic sums.
\begin{multline*}\sum_{i=1}^{n-2}x^{i+1}(1+x^{i+2})\dots(1+x^{n-1})=
\sum_{i=1}^{n-2}(1+x^{i+1})(1+x^{i+2})\dots(1+x^{n-1})-(1+x^{i+2})\dots(1+x^{n-1})\\
=(1+x_2)\dots(1+x_{n-1})-1,
\end{multline*}
\begin{multline*}\sum_{i=1}^{n-2}x_1\dots x_{i}(1+x_{i+2})\dots(1+x_{n-1})\\
=\sum_{i=1}^{n-2}x_1\dots x_{i}(1+x_{i+1})(1+x_{i+2})\dots(1+x_{n-1})-x_1\dots x_{i}x_{i+1}(1+x_{i+2})\dots(1+x_{n-1})\\
=x_1(1+x_2)\dots(1+x_{n-1})-x_1\dots x_{n-1}.\end{multline*}

Plugging these formulas back into equation~(\ref{eq:Tnsum}) we get
\begin{eqnarray*}\sum_{T\in\T_n}\x^{\D(T)}&=&\sum_{i=1}^{n-2}(1+x_1)\dots(1+x_{i-1})(x_i+x_{i+1})(1+x_{i+2})\dots(1+x_{n-1})\\
&&\qquad -(1+x_1)(1+x_2)\dots(1+x_{n-1})+1+x_1\dots x_{n-1}\\
&=&\sum_{\pi\in\A_n}\x^{\D(\pi)}-2(1+x_1)(1+x_2)\dots(1+x_{n-1})+1+x_1\dots
x_{n-1}=\sum_{\pi\in\Z_n}\x^{\D(\pi)}.
\end{eqnarray*}
\end{proof}

\subsection{A bijective proof}\label{sec:bijZT}

We now give a bijection $\phi$ between
$\Z_n$ and $\T_n$ that preserves the descent set, providing an alternative proof of Theorem~\ref{thm:TZ}.

Given $\pi\in\Z_n$ with $n\ge4$, consider two cases. If
$\pi^{-1}(1)>\pi^{-1}(n)$ (this happens iff $\pi$ can be drawn on
the picture on the left of Figure~\ref{fig:grid_arc}, and also iff
$\pi(1)>\pi(n)$), let $j=\pi^{-1}(1)$, and let
$$I=\{i:\pi(i)\ge\pi(1)\}\cup\{i:i>j+1\mbox{ and }\pi(i-1)<\pi(n)\}.$$
Then $\phi(\pi)\in\T_n$ is the tableau having the elements of $I$
in the first row, having $j+1$ in the box in the second row and
column, and having the rest of the elements of $[n]=\{1,2,\dots,n\}$ in the first
column.

If $\pi^{-1}(1)<\pi^{-1}(n)$ (this happens iff $\pi$ can be drawn
on the picture on the right of Figure~\ref{fig:grid_arc}, and also
iff $\pi(1)<\pi(n)$), let $j=\pi^{-1}(n)$, and let
$$I=\{i:\pi(i)\le\pi(1)\}\cup\{i:i>j+1\mbox{ and }\pi(i-1)>\pi(n)\}.$$
Then $\phi(\pi)\in\T_n$ is the tableau having the elements of $I$
in the first column, having $j+1$ in the box in the second row and
column, and having the rest of the elements of $[n]$ in the first
row.

For example, if $$\pi=\left(\begin{array}{ccccccccccc}
                         1 & 2 & 3 & 4 & 5 & 6 & 7 & 8 & 9 & 10 & 11 \\
                         8 & 9 & 10 & 7 & 11 & 1 & 2 & 6 & 5 & 3 & 4 \\
                       \end{array}
                     \right),$$ then $j=6=\pi^{-1}(1)>\pi^{-1}(n)=5$,
$I=\{1,2,3,5,8,11\}$, and $\phi(\pi)$ is the tableau
$$\young(12358\eleven,47,6,9,\ten).$$ If
$$\pi=\left(\begin{array}{cccccccccccc}
                         1 & 2 & 3 & 4 & 5 & 6 & 7 & 8 & 9 & 10 & 11 & 12\\
                         2 & 3 & 1 & 4 & 5 & 6 & 12 & 7 & 11 & 10 & 8 & 9 \\
                       \end{array}
                     \right),$$ then $j=7=\pi^{-1}(n)>\pi^{-1}(1)=3$,
$I=\{1,3,10,11\}$, and $\phi(\pi)$ is the tableau
$$\young(1245679\twelve,38,\ten,\eleven).$$

In the case $\pi^{-1}(1)>\pi^{-1}(n)$, the descent set of both $\pi$ and $\phi(\pi)$ equals $\{i:i+1\notin I,\,i\ne j\}$.
In the case $\pi^{-1}(1)<\pi^{-1}(n)$, the descent set of $\pi$ and $\phi(\pi)$ equals $\{i:i+1\notin I\}\cup\{j\}$.
To see that $\phi$ is a bijection, note that given a tableau in $T\in\T_n$ with entry $j+1$ in the box in the second row and column, 
we can distinguish the two cases by checking whether $j$ is in the first row (case $\pi^{-1}(1)<\pi^{-1}(n)$) or not (case $\pi^{-1}(1)>\pi^{-1}(n)$).
In both cases, the set $I$ can be immediately recovered from the tableau, and together with the value of $j$, it uniquely determines the permutation $\phi^{-1}(T)$.

\subsection{A shape-preserving bijection}

Here we give another bijection $\psi$ between $\Z_n$ and $\T_n$.
It has the property that $\psi(\pi)$ has the same shape as $\phi(\pi)$,
although it does not preserve the descent set. Another property is that the entries in the tableau $\psi(\pi)$
that are weakly north and strictly east of $2$ are the elements in
the set $C(\pi)$, defined as follows.

\begin{definition}
For $\pi\in\A_n$, let $C(\pi)$ be the set of values $i\in\{3,4,\dots,n\}$ such that
$\ol{\pi(i-1)-1}\in\{\pi(1),\pi(2),\dots,\pi(i-2)\}$.
\end{definition}

Let $\pi\in\Z_n$ with $n\ge4$. If $\pi^{-1}(1)>\pi^{-1}(n)$, let
$j=\pi^{-1}(1)$, and let
$$S=\{1\}\cup\{i+1:\pi(1)\ge \pi(i)> \pi(n)\}=[n]\setminus C(\pi).$$
Then $\psi(\pi)$ is the tableau whose
first column is $S$, having $j+1$ in the box in the second row and
column, and having the rest of the elements in $[n]$ in the first
row.

If $\pi^{-1}(1)>\pi^{-1}(n)$, let $j=\pi^{-1}(n)$, and let
$$S=\{1\}\cup\{i+1:\pi(1)\le \pi(i)<\pi(n)\}=\{1,2\}\cup C(\pi).$$
Then $\psi(\pi)$ is the tableau
whose first row is $S$, having $j+1$ in the box in the second row
and column, and having the rest of the elements in $[n]$ in the
first column.

For example, if $$\pi=\left(\begin{array}{cccccccccccc}
                         1 & 2 & 3 & 4 & 5 & 6 & 7 & 8 & 9 & 10 & 11 & 12\\
                         10 & 9 & 11 & 8 & 7 & 12 & 1 & 2 & 6 & 5 & 3 & 4 \\
                       \end{array}
                     \right),$$ then $j=7=\pi^{-1}(1)$,
$S=\{1,2,3,5,6,10,11\}$, and $\psi(\pi)$ is the tableau
$$\young(1479\twelve,28,3,5,6,\ten,\eleven).$$ If
$$\pi=\left(\begin{array}{ccccccccccc}
                         1 & 2 & 3 & 4 & 5 & 6 & 7 & 8 & 9 & 10 & 11 \\
                         3 & 2 & 4 & 1 & 5 & 6 & 11 & 7 & 10 & 9 & 8 \\
                       \end{array}
                     \right),$$ then $j=7=\pi^{-1}(n)$,
$S=\{1,2,4,6,7,9\}$, and $\psi(\pi)$ is the tableau
$$\young(124679,38,5,\ten,\eleven).$$

The map $\psi$ is clearly invertible, since given a tableau $T\in\T_n$, the position of the entry $2$ determines which of the cases $\pi^{-1}(1)>\pi^{-1}(n)$ or $\pi^{-1}(1)>\pi^{-1}(n)$ we are in.
In both cases, the set $S$ and the value of $j$, which are immediately recovered from the tableau, uniquely determine the permutation $\psi^{-1}(T)$. The fact that $\phi(\pi)$ and $\psi(\pi)$ have the same shape
follows by noting that in both cases $|I|+|S|=n$, with $I$ as defined in Subsection~\ref{sec:bijZT}, and thus the tableaux $\phi(\pi)$ and $\psi(\pi)$, which consist of a hook plus a box, have the same number of rows.

\section{Encoding by descents}\label{sec:encoding}

In this section we present an encoding of arc permutations that is
different from the one used in Section~\ref{X-graph}. This
encoding, which keeps track of the positions of the descents, is applied to prove further enumerative results,
some of which will be used in Section~\ref{sec:regev}.

\subsection{The encoding}\label{sec:encoding-by-descents}

Let $\WW_n$ be the set of words $\ww=w_1w_2\dots w_{n-1}$ over the
alphabet $\{A,D\}$ where at most one adjacent pair $AD$ or $DA$
may be underlined. We encode arc permutations by words in $\WW_n$ via the following bijection $\enc$.

\begin{lemma}\label{descents-encoding}
There is a bijection $\enc: \A_n \longrightarrow \WW_n$ such that if $\enc(\pi)=\ww$, then
$$
\ww_i=D \ \Longleftrightarrow \ i\in \D(\pi).
$$
\end{lemma}

\begin{proof}
Given $\pi\in\A_n$, first let $w_i=D$ if $i\in\D(\pi)$, and $w_i=A$
otherwise, for each $1\le i\le n-1$. If $\pi\in\LL_n$, then $\enc(\pi)$
is the word $\ww$ with no underlined pair.
Otherwise, let $k$ be the smallest such that
$\{\pi(1),\dots,\pi(k)\}$ is not an interval in $\bbz$, and note
that $k<n$. If $w_{k-1}=D$, then $\pi(k)=1<\pi(k+1)$, so $w_k=A$.
Similarly, if $w_{k}=A$, then $\pi(k)=n>\pi(k+1)$, so $w_k=D$. In
both cases, $\enc(\pi)$ is the word $\ww$ with the pair $w_{k-1}w_{k}$ underlined.

To show that $\enc$ is a bijection between $\A_n$ and
$\WW_n$, we describe the inverse map. Given $\ww\in\WW_n$, let
$w_{k-1} w_k$ be the underlined pair in $\ww$ if there is one, and let $k=n+1$ otherwise.
It is easy to verify that the unique $\pi\in\A_n$ with encoding
$\enc(\pi)=\ww$ can be recovered as follows.

If $w_{k-1} w_k=\ul{DA}$, let $\delta=n+1$, otherwise let $\delta=k$.
Then, for $1\le i< k$,
$$
\pi(i)=\begin{cases} \delta-1-|\{j\in[i,k-2]:\, w_j=A\}|& \text{ if $i=1$ or $w_{i-1}=A$},\\
\delta-k+1+|\{j\in[i,k-2]:\, w_j=D\}|& \text{ if $w_{i-1}=D$}.
\end{cases}
$$
Now let $\delta'=\delta\mod n$. For $k\le i\le n$,
$$
\pi(i)=\begin{cases} \delta'+|\{j\in[k,i-1]:\, w_j=A\}|& \text{ if $i=1$ or $w_{i}=A$},\\
\delta'+n-k-|\{j\in[k,i-1]:\, w_j=D\}|& \text{ if $w_{i}=D$}.
\end{cases}
$$
\end{proof}

For example, $\enc(342561)=ADAAD$, $\enc(12543)=A\ul{AD}D$, and $\enc(65781423)=DAA\ul{DA}DA$.
 The following is
an immediate consequence of our encoding of arc permutations.

\begin{proposition}
Let $B\subseteq [n-1]$. Then
$$|\{\pi\in\A_n : \D(\pi)=B\}|=1+|\{i\in[n-2]:\,|B\cap\{i,i+1\}|=1\}|.$$
\end{proposition}

\begin{proof}
Permutations in $\A_n$ with descent set $B$ correspond via $\enc$ to
encodings $\ww\in\WW_n$ such that $w_i=D$ if and only if $i\in B$.
There is one such encoding with no underlined pairs, and one
encoding where the pair $w_iw_{i+1}$ is underlined for each $i$
for which this pair equals $AD$ or $DA$.
\end{proof}

\subsection{$\mu$-left-unimodal permutations}\label{sec:muleft}

Here we define a generalization of left-unimodal permutations.
Let $\mu=(\mu_1,\dots,\mu_t)$ be a partition of $n$ with $t$
nonzero parts. For $1\le i\le t$, denote
$$
\mu_{(i)}:=\sum\limits_{j=1}^i \mu_j
$$
and
$$
S(\mu):=(\mu_{(1)},\dots,\mu_{(t)}).
$$
Let $\mu_{(0)}:=0$. A permutation $\pi \in \S_n$ is {\em
$\mu$-left-unimodal} if for every $0\le i< t$ there exists $0\le
j_i \le \mu_{i+1}$ such that
$$
\pi^{-1}(\mu_{(i)}+1)>\pi^{-1}(\mu_{(i)}+2)>\cdots >
\pi^{-1}(\mu_{(i)}+j_i)< \pi^{-1}(\mu_{(i)}+j_i+1)<\cdots<
\pi^{-1}(\mu_{(i+1)}).
$$
Denote the set of $\mu$-left-unimodal permutations by $\LL_\mu$.
Note that for the partition with one part $\mu=(n)$, we have
$\LL_{(n)}=\LL_n$, the set of left-unimodal permutations in
$\S_n$.

We will consider the set $\LL_\mu^{-1}$, consisting of those
permutations whose inverse is $\mu$-left-unimodal. Translating the
above definition, we see that $\pi\in\LL_\mu^{-1}$ if for every $0\le i< t$,
the sequence
$\pi(\mu_{(i)}+1),\pi(\mu_{(i)}+2),\dots,\pi(\mu_{(i+1)})$ first
decreases and then increases; we say that this sequence is {\em $V$-shaped}.

For example, if $\mu=(\mu_1,\mu_2,\mu_3)=(4,3,1)$, then
$S(\mu)=(\mu_{(1)},\mu_{(2)},\mu_{(3)})=(4,7,8)$. In this case
$53687142,35687412\in\LL_\mu^{-1}$ because the sequences $5368$,
$714$, $8$, $3568$, $741$, $2$  are $V$-shaped. However,
$53867142,53681742\notin\LL_\mu^{-1}$ because the sequences $5286$
and $174$ are not $V$-shaped.

\subsection{Enumeration of arc permutations whose
inverse is $\mu$-left-unimodal}\label{sec:regev-unsugned}

\begin{proposition}\label{prop:enumAL}
For every partition $\mu=(\mu_1,\ldots,\mu_r, 1^s)$ of $n$ with
$\mu_r>1$,
$$
|\A_n \cap
\LL_\mu^{-1}|=\mu_1\dots\mu_r\,2^{r+s}\left(r+\frac{s}{4}-\sum_{i=1}^r
\frac{1}{\mu_i}\right).
$$
\end{proposition}

Note that when $\mu=(1^n)$, Proposition~\ref{prop:enumAL} gives Claim~3. %\ref{arc-enumeration}.
Indeed, in this case $\LL_{(1^n)}^{-1}=\S_n$, $r=0$ and $s=n$, so
$$
|\A_n|=|\A_n \cap \LL_{(1^n)}^{-1}|=2^n\left(\frac{n}{4}\right) =
n2^{n-2}.
$$

\begin{proof}
We use the encoding $\enc$ of permutations in $\A_n$ by words in $\WW_n$,
defined in Lemma~\ref{descents-encoding}.
Recall that $\pi\in\A_n \cap \LL_\mu^{-1}$ if
$\pi(\mu_{(i-1)}+1),\pi(\mu_{(i-1)}+2),\dots,\pi(\mu_{(i)})$ is
$V$-shaped for all $1\le i\le r+s$. By Lemma~\ref{descents-encoding},
this property is equivalent to the fact that for all
$1\le i\le r$, the subword
$\ww^{(i)}:=w_{\mu_{(i-1)}+1}w_{\mu_{(i-1)}+2}\dots w_{\mu_{(i)}-1}$
has no $A$ followed by a $D$.

To find $|\A_n \cap \LL_\mu^{-1}|$ we will enumerate the words $\ww\in\WW_n$
where each block $\ww^{(i)}$ satisfies this condition. Note
that the length of $\ww^{(i)}$ is $\mu_i-1$, and that
$\ww$ has one letter between each block and the next, and $s$
letters to the right of the rightmost block $\ww^{(r)}$. Consider
four cases as follows.
\ben\renewcommand{\labelenumi}{(\roman{enumi})} \item If $\ww$ has
no underlined pair, then $\ww^{{(i)}}=D^kA^{\mu_i-k-1}$ for some
$0\le k\le \mu_i-1$, so there are $\mu_i$ choices for each
$\ww^{{(i)}}$, times $2^{r+s-1}$ choices for the remaining
letters, for a total of $\mu_1\mu_2\dots\mu_r2^{r+s-1}$ words of
this form. \item If $\ww$ has an underlined pair contained inside
a block, say $\ww^{{(i)}}$, then
$\ww^{{(i)}}=D^k\ul{DA}A^{\mu_i-k-3}$ for some $0\le k\le
\mu_i-3$, so there are $\mu_i-2$ choices for this block. The
remaining choices are as in case (i), so the number of words of
this form is
$$\sum_{i=1}^r \mu_1\dots\mu_{i-1}(\mu_i-2)\mu_{i+1}\dots\mu_r2^{r+s-1}
=\mu_1\mu_2\dots\mu_r2^{r+s-1}\left(r-2\sum_{i=1}^r\frac{1}{\mu_i}\right).$$
\item If $\ww$ has an underlined pair outside of the blocks
$\ww^{{(i)}}$, then there are $s-1$ choices for the location of
the pair (assuming $s\ge1$), $2$ choices for whether it is
$\ul{AD}$ or $\ul{DA}$, and $2^{r+s-3}$ choices for the remaining
letters outside the blocks, giving a total of
$\mu_1\mu_2\dots\mu_r(s-1)2^{r+s-2}$ words of this form. When
$s=0$ there are no words of this form. \item If $\ww$ has an
underlined pair that is partly inside a block $\ww^{{(i)}}$ and
partly outside, the number of choices for the location of the
underlined pair is $2r-1$ if $s\ge1$, since it can be at the
beginning or at the end of any block but not at the beginning of
$\ww^{(1)}$, and $2r-2$ if $s=0$, since it cannot be at the end of
block $\ww^{(r)}$ either. Afterwards, the letters inside the
blocks can be chosen in $\mu_1\mu_2\dots\mu_r$ ways as before,
with the understanding that the choice of the underlined letter
forces the other underlined letter. Now we have $2^{r+s-2}$
choices for the not underlined letters outside the blocks, for a
total of $\mu_1\mu_2\dots\mu_r(2r-1)2^{r+s-2}$ words if $s\ge1$,
or $\mu_1\mu_2\dots\mu_r(2r-2)2^{r+s-2}$ if $s=0$. \een Adding the
four contributions we obtain the stated formula.
\end{proof}

Next we give a signed version of Proposition~\ref{prop:enumAL}, that is, a formula for signed enumeration of arc permutations whose
inverse is $\mu$-left-unimodal. This formula will be
used in the proof of Regev's character formula in Section~\ref{sec:regev}.

\begin{proposition}\label{arc-mu-unimodal}
For every partition $\mu=(\mu_1,\ldots,\mu_r, 1^s)$ of $n$ with
$\mu_r>1$, \beq\label{eq:altsum} \sum\limits_{\pi\in \A_n \cap
\LL_\mu^{-1}} (-1)^{|\D(\pi)\setminus S(\mu)|}= \frac{1}{4}\cdot s
\cdot \prod\limits_{j=1}^{r+s} (1+(-1)^{\mu_j-1}). \eeq
\end{proposition}

Before proving this proposition, we give an example for $n=4$ and $\mu=(3,1)$. In this case,
$$\A_n \cap \LL_\mu^{-1}=\{1234,1243,2134,2143,2341,3214,3241,4123,4132,4312,4321\}.$$
Since $|\D(\pi)\setminus S(\mu)|=|\D(\pi)\setminus\{3\}|$, the
left hand side of equation~\eqref{eq:altsum} becomes
$$(-1)^0+(-1)^0+(-1)^1+(-1)^1+(-1)^0+(-1)^2+(-1)^1+(-1)^1+(-1)^1+(-1)^2+(-1)^2=1.$$

\begin{proof}[Proof of Proposition~\ref{arc-mu-unimodal}]
We use the encoding $\enc$ from Lemma~\ref{descents-encoding}.
Defining the sign of a permutation $\pi$ to be
$(-1)^{|\D(\pi)\setminus S(\mu)|}$, we will construct a
sign-reversing involution on the set $\A_n \cap \LL_\mu^{-1}$. If
$\mu$ has some part of even size, then this involution will have
no fixed points, so all the terms in the left hand side
of~\eqref{eq:altsum} cancel with each other. If all the parts of
$\mu$ have odd size, then some permutations will not be canceled by
the involution, but their contribution to the sum is easy to
compute.

Recall that $\pi\in\A_n \cap \LL_\mu^{-1}$ if and only if
$\ww^{(i)}=w_{\mu_{(i-1)}+1}w_{\mu_{(i-1)}+2}\dots w_{\mu_{(i)}-1}$
has no $A$ followed by a $D$ for all $1\le i\le r$.

Suppose first that $\mu$ has some part of even size, and let
$\mu_e$ be the first such part. We define an involution $\invol$
on the set of encodings $\ww$ of permutations in $\A_n \cap \LL_\mu^{-1}$
by changing only the subword $\ww^{(e)}$, and possibly the letters
immediately preceding and following $\ww^{(e)}$, but leaving the
rest of the word $\ww$ unchanged. Note that $\ww^{(e)}$ has
odd length, which we write as $\mu_e-1=2a+1$. If the underlined
pair is completely inside or outside of $\ww^{(e)}$ (or there is
no underlined pair), define $\invol$ on $\ww^{(e)}$ as follows:
\begin{align*}
& D^{2i+1}A^{2(a-i)}\overset{\invol}{\longleftrightarrow} D^{2i}A^{2(a-i)+1} && \mbox{for } 0\le i\le a,\\
& D^{2i+1}\ul{DA}A^{2(a-i-1)}\overset{\invol}{\longleftrightarrow}
D^{2i}\ul{DA}A^{2(a-i-1)+1} && \mbox{for } 0\le i\le a-1.
\end{align*}
If the pair  $w_{\mu_{(e)}-1}w_{\mu_{(e)}}$ is underlined, define
$\invol$ on $\ww^{(e)}w_{\mu_{(e)}}$ as follows:
\begin{align*}
& D^{2a}\ul{DA}\overset{\invol}{\longleftrightarrow} D^{2a}\ul{AD}, &&\\
& D^{2i+1}A^{2(a-i)-1}\ul{AD}\overset{\invol}{\longleftrightarrow}
D^{2i}A^{2(a-i)}\ul{AD} && \mbox{for } 0\le i\le a-1.
\end{align*}
Finally, if the pair $w_{\mu_{(e-1)}}w_{\mu_{(e-1)}+1}$ is
underlined, define $\invol$ on $w_{\mu_{(e-1)}}\ww^{(e)}$ as
follows:
\begin{align*}
& \ul{AD}A^{2a}\overset{\invol}{\longleftrightarrow} \ul{DA}A^{2a},&&\\
& \ul{AD}D^{2i}A^{2(a-i)}\overset{\invol}{\longleftrightarrow}
\ul{AD}D^{2i-1}A^{2(a-i)+1} && \mbox{for } 1\le i\le a.
\end{align*}

Since the above definitions cover all the possibilities, the map
$\invol$ is an involution on the set $\enc(\A_n \cap \LL_\mu^{-1})$. It is also clear that the sign of the
permutation encoded by a word changes when applying $\invol$ to
it, since the parity of the number of $D$s in $\ww^{(e)}$ changes
(note that the letters $w_{\mu_{(e-1)}}$ and $w_{\mu_{(e)}}$ do
not contribute to the sign). It follows that $$\sum\limits_{\pi\in
\A_n \cap \LL_\mu^{-1}} (-1)^{|\D(\pi)\setminus S(\mu)|}=0$$ when
$\mu$ has some part of even size, which agrees with the right hand
side of equation~\eqref{eq:altsum} in this case.

Consider now the case where all the parts of $\mu$ have odd size.
The case $r=0$ is trivial because then $\A_n \cap
\LL_\mu^{-1}=\A_n$ and $|\D(\pi)\setminus S(\mu)|=0$, so
equation~\eqref{eq:altsum} becomes $|\A_n|=n2^{n-2}$, which is
true by Claim~3. %\ref{arc-enumeration}.
Suppose in what follows that
$r\ge1$. For each fixed $1\le i\le r$, we first define an operation
$\invol_i$ on $\ww^{{(i)}}$ and its surrounding letters that can
only be applied if this subword has a certain form.
Write the length of $\ww^{{(i)}}$, which is even, as $\mu_i-1=2b$.
If the underlined pair is completely inside or
outside of $\ww^{{(i)}}$, or there is no underlined pair, define
$\invol_i$ on $\ww^{{(i)}}$ as follows:
\begin{align*}
& D^{2i+1}A^{2(b-i)-1}\overset{\invol_i}{\longleftrightarrow} D^{2i}A^{2(b-i)} && \mbox{for } 0\le i\le b-1,\\
&
D^{2i+1}\ul{DA}A^{2(b-i-1)-1}\overset{\invol_i}{\longleftrightarrow}
D^{2i}\ul{DA}A^{2(b-i-1)} && \mbox{for } 0\le i\le b-2.
\end{align*}
If $w_{\mu_{(i)}-1}w_{\mu_{(i)}}=\ul{AD}$, define $\invol_i$ on
$\ww^{{(i)}}w_{\mu_{(i)}}$ by
\begin{align*}
&
D^{2i+1}A^{2(b-i-1)}\ul{AD}\overset{\invol_i}{\longleftrightarrow}
D^{2i}A^{2(b-i-1)+1}\ul{AD} && \mbox{for } 0\le i\le b-1.
\end{align*}
If $w_{\mu_{(i-1)}}w_{\mu_{(i-1)}+1}=\ul{AD}$, define $\invol_i$
on $w_{\mu_{(i-1)}}\ww^{{(i)}}$ by
\begin{align*}
&
\ul{AD}D^{2i}A^{2(b-i-1)+1}\overset{\invol_i}{\longleftrightarrow}
\ul{AD}D^{2i+1}A^{2(b-i-1)} && \mbox{for } 0\le i\le b-1.
\end{align*}
The operation $\invol_i$ is defined for all words $\ww\in\WW_n$
except when $\ww^{{(i)}}=D^{2b}$, $\ww^{{(i)}}=D^{2b-2}\ul{DA}$,
$\ww^{{(i)}}w_{\mu_{(i)}}= D^{2b-1}\ul{DA}$, or
$w_{\mu_{(i-1)}}\ww^{{(i)}}=\ul{DA}A^{2b-1}$.

Now we are ready to define the involution $\invol$ when all the parts of $\mu$ are odd. To compute
$\invol(\ww)$ for a given $\ww\in\enc(\A_n \cap \LL_\mu^{-1})$, if there is some $i$ with
$1\le i\le r$ for which the operation $\invol_i$ is defined, take
the smallest such $i$ and apply $\invol_i$ to $\ww^{{(i)}}$ and
its surrounding letters, leaving the rest of $\ww$ unchanged. If
there is no such $i$, then $\invol$ does not change $\ww$, that
is, $\ww$ is a fixed point of $\invol$. It is clear from the
construction that $\invol$ is a sign-reversing involution, since
$\invol_i$ changes the parity of the number of $D$s in
$\ww^{{(i)}}$, while the other $\ww^{{(j)}}$ for $j\ne i$ are
unchanged. However, unlike in the case of parts of even size, here
$\invol$ has some fixed points: those words $\ww\in\WW_n$ for
which the operation $\invol_i$ is not defined for any $1\le i\le
r$. For this to happen, $\ww^{{(i)}}$ has to equal either
$D^{\mu_i-1}$, $D^{\mu_i-3}\ul{DA}$, $D^{\mu_i-2}\ul{D}$, or
$\ul{A}A^{\mu_i-2}$, for all $1\le i\le r$.

Let us enumerate the words $\ww\in\WW_n$ having this property,
separating them in four cases. Note that $\ww$ has one letter
between $\ww^{(i)}$ and $\ww^{(i+1)}$ for each $1\le i\le r-1$,
and $s$ letters to the right of $\ww^{(r)}$. Suppose first that
$s\ge1$. \ben\renewcommand{\labelenumi}{(\roman{enumi})} \item If
$\ww$ has no underlined pair, then $\ww^{{(i)}}=D^{\mu_i-1}$ for
all $i$, and there are $2^{r+s-1}$ choices for the remaining
letters of $\ww$. The corresponding permutations have positive
sign, since every block $\ww^{{(i)}}$ has an even number of $D$s.
\item If $\ww$ has an underlined pair contained inside a block,
the above number of choices has to be multiplied by the $r$
choices of the block for which $\ww^{{(i)}}=D^{\mu_i-3}\ul{DA}$.
The $r\,2^{r+s-1}$ corresponding permutations have now negative
sign, since exactly one block $\ww^{{(i)}}$ has an odd number of
$D$s. \item If $\ww$ has an underlined pair outside of all the
blocks $\ww^{{(i)}}$, then again $\ww^{{(i)}}=D^{\mu_i-1}$ for all
$i$. Since the underlined pair has to be contained in the last $s$
letters of $\ww$, we have $s-1$ choices for its location, times
$2$ choices for whether it is $\ul{AD}$ or $\ul{DA}$, and finally
$2^{r+s-3}$ choices for the remaining letters in $\ww$. The
$(s-1)2^{r+s-2}$ corresponding permutations have positive sign.
\item If $\ww$ has an underlined pair that is partly inside a
block $\ww^{{(i)}}$ and partly outside, we have to choose the
index $i$ such that $\ww^{{(i)}}$ equals $D^{\mu_i-2}\ul{D}$ or
$\ul{A}A^{\mu_i-2}$. For each possible $i$ we have two choices,
except for $i=1$, in which case $\ww^{{(1)}}=D^{\mu_1-2}\ul{D}$ is
the only possibility because there is no letter to the left of
$\ww^{{(1)}}$. This gives $2r-1$ choices for the underlined pair,
which also forces the entries in all the blocks $\ww^{{(j)}}$,
leaving $2^{r+s-2}$ choices for the remaining entries of $\ww$.
The $(2r-1)2^{r+s-2}$ corresponding permutations have again
positive sign. \een Adding the contributions in the four cases,
the sum of the signs of the permutations fixed by the involution
is \beq\label{eq:rs}
2^{r+s-1}-r\,2^{r+s-1}+(s-1)2^{r+s-2}+(2r-1)2^{r+s-2}=s\,2^{r+s-2},\eeq
which agrees with the right hand side of~\eqref{eq:altsum} when
all the parts of $\mu$ have odd size. Formula~\eqref{eq:rs} also
holds when $s=0$. The only change in the argument is that there
are no permutations in case (iii), but in case (iv), when choosing
the index $i$ such that $\ww^{{(i)}}$ equals $D^{\mu_i-2}\ul{D}$
or $\ul{A}A^{\mu_i-2}$, the choice $i=r$ forces
$\ww^{{(r)}}=\ul{A}A^{\mu_r-2}$, giving $2r-2$ choices for the
underlined pair, times $2^{r-2}$ choices for the remaining entries
of $\ww$.
\end{proof}

As an example of the involution in the above proof, let $n=10$ and
$\mu=(5,3,1,1)$. The words fixed by $\invol$ of each of the above
types are those of the following form:
\ben\renewcommand{\labelenumi}{(\roman{enumi})} \item $DDDD\,
w_5\, DD\,w_8\,w_9$, \item $DD\ul{DA}\, w_5\, DD\,w_8\,w_9$,
$DDDD\, w_5\, \ul{DA}\,w_8\,w_9$, \item $DDDD\, w_5\,
DD\,\ul{AD}$, $DDDD\, w_5\, DD\,\ul{DA}$, \item $DDD\ul{DA}\,
DD\,w_6\,w_7$, $DDDD\, \ul{DA}A\,w_6\,w_7$, $DDDD\, w_5\,
D\ul{DA}\, w_7$. \een The sum of the signs of the corresponding permutations is
$8-16+4+12=8=s\,2^{r+s-2}$. The contributions of all the other
words are canceled by the sign-reversing involution. For example,
$$\invol(DDAA\, w_5w_7\,w_8\,w_9)=DDDA\, w_5w_7\,w_8\,w_9 \quad \mbox{and} \quad \invol(DDDD\, w_5\, D\ul{AD}A)=DDDD\, w_5\, A\ul{AD}A.$$

\section{A character formula of Regev}\label{sec:regev}

In this section we apply Theorem~\ref{thm:TZ} and Proposition~\ref{arc-mu-unimodal}
to prove a conjectured character formula of Amitai Regev.

Let $V$ be an $(n-1)$-dimensional vector space over $\CC$, and let $\wedge V$ be
its exterior algebra. Consider the natural action of the symmetric
group $\S_{n-1}$ on $\wedge V$, and denote the character of the
induced $\S_{n}$-module $\wedge V\uparrow^{\S_n}$ by $\chi_n$.
Regev conjectured the following character formula~\cite{Regev}.

\begin{theorem}[{\bf Regev's Formula}]\label{thm:regev}
Let $\mu=(\mu_1,\ldots,\mu_r, 1^s)$ be a partition of $n$ with $\mu_r>1$. Then
$$\chi_n(\mu)=\frac{1}{4}\cdot s \cdot \prod\limits_{j=1}^{r+s} (1+(-1)^{\mu_j-1}).$$
\end{theorem}

Let us introduce some notation for the proof of Regev's Formula.
Given a partition $\mu=(\mu_1,\dots,\mu_t)$ of $n$, recall from Section~\ref{sec:muleft} that $\LL_{\mu}^{-1}$ is
the set of permutations whose inverse is $\mu$-left-unimodal.
If $\RSK(\pi)=(P,Q)$, then the fact that $\pi\in\LL_{\mu}^{-1}$ translates into the following condition on the descent set of $Q$:
for every $1\le i< t$ there exists $0\le j_i \le \mu_{i+1}$ such that $\mu_{(i)}+k\in \D(Q)$ for $1\le k <j_i$, and $\mu_{(i)}+k\not\in \D(Q)$ for $j_i\le k< \mu_{i+1}$.
A standard Young tableau $Q$ satisfying this condition is called {\em $\mu$-unimodal}. Denote the set of $\mu$-unimodal standard Young tableaux of
shape $\lambda$ by ${\rm SYT}_\mu^\lambda$.

\begin{proof}[Proof of Theorem~\ref{thm:regev}]
It is well known that the exterior algebra $\wedge V$ is
equivalent as an $\S_{n-1}$-module to a direct sum of all Specht
modules indexed by hooks, see e.g.~\cite[Ex. 4.6]{FH}. By the
branching rule, the decomposition of the induced exterior algebra
$\wedge V\uparrow^{\S_n}$ into irreducibles is then given by
\begin{equation}\label{decomposition}
\chi_n=
\sum\limits_{k=1}^{n} \chi^{(k,1^{n-k})}+ \sum\limits_{k=2}^{n-1}
\chi^{(k,1^{n-k})}+\sum\limits_{k=2}^{n-2} \chi^{(k,2,1^{n-k-2})}.
\end{equation}

Now we use the following character formula for symmetric group irreducible characters, which is a
special case of~\cite[Theorem 4]{Ro2}, see also~\cite{Ram,RIMS}:
$$\chi^\lambda(\mu)=\sum\limits_{T\in {\rm SYT}_\mu^\lambda}(-1)^{|\D(T)\setminus S(\mu)|}.$$
 Applying this formula to equation~\eqref{decomposition},
we get the following expression for the character of $\chi_n$ at a conjugacy class of cycle type $\mu$:
\begin{multline}\label{eq:3sums}
\chi_n(\mu)=
 \sum\limits_{k=1}^n \sum\limits_{T\in {\rm
SYT}_\mu^{(k,1^{n-k})}}(-1)^{|\D(T)\setminus S(\mu)|}+
 \sum\limits_{k=2}^{n-1} \sum\limits_{T\in {\rm
SYT}_\mu^{(k,1^{n-k})}}(-1)^{|\D(T)\setminus S(\mu)|}\\
+\sum\limits_{k=2}^{n-2} \sum\limits_{T\in {\rm
SYT}_\mu^{(k,2,1^{n-k-2})}}(-1)^{|\D(T)\setminus S(\mu)|}.
\end{multline}
By Remark~\ref{knuth1}, the RSK correspondence gives a descent-set-preserving bijection between permutations $\pi\in\LL_n$ and standard Young tableaux $Q$ of hook shape.
Since $\pi\in\LL_\mu^{-1}$ if and only if $Q$ is $\mu$-unimodal, we get
$$
 \sum\limits_{k=1}^n \sum\limits_{T\in {\rm
SYT}_\mu^{(k,1^{n-k})}}(-1)^{|\D(T)\setminus
S(\mu)|}=\sum\limits_{\pi\in \LL_n\cap \LL_\mu^{-1}}
(-1)^{|\D(\pi)\setminus S(\mu)|}.
$$
Combining Remarks~\ref{knuth1} and~\ref{knuth2}, the RSK correspondence also gives a descent-set-preserving bijection between permutations $\pi\in\U_n\setminus\LL_n$ and standard Young tableaux $Q$ of hook shape
having at least two rows or two columns. It follows that
$$
\sum\limits_{k=2}^{n-1} \sum\limits_{T\in {\rm
SYT}_\mu^{(k,1^{n-k})}}(-1)^{|\D(T)\setminus
S(\mu)|}=\sum\limits_{\pi\in (\U_n\setminus \LL_n)\cap
\LL_\mu^{-1}} (-1)^{|\D(\pi)\setminus S(\mu)|}.
$$
For the third sum in~\eqref{eq:3sums}, instead of the RSK correspondence, we use Theorem~\ref{thm:TZ}, which was proved in Section~\ref{sec:bijZT} via a different descent-set-preserving bijection. Standard Young tableaux in $\T_n$ that are $\mu$-unimodal
correspond to permutations in $\Z_n$ whose inverse is $\mu$-left-unimodal, and so
$$
\sum\limits_{k=2}^{n-2} \sum\limits_{T\in {\rm
SYT}_\mu^{(k,2,1^{n-k-2})}}(-1)^{|\D(T)\setminus
S(\mu)|}=\sum\limits_{\pi\in \Z_n\cap \LL_\mu^{-1}}
(-1)^{|\D(\pi)\setminus S(\mu)|}.
$$
Combining the last four equations and using that $\A_n$ is the disjoint union of $\LL_n$, $\U_n\setminus\LL_n$ and $\Z_n=\A_n\setminus\U_n$, we get
\begin{align*}
\chi_n(\mu) & = \sum\limits_{\pi\in \LL_n\cap \LL_\mu^{-1}}
(-1)^{|\D(\pi)\setminus S(\mu)|}+
 \sum\limits_{\pi\in (\U_n\setminus \LL_n)\cap
\LL_\mu^{-1}} (-1)^{|\D(\pi)\setminus S(\mu)|}+
\sum\limits_{\pi\in \Z_n\cap \LL_\mu^{-1}} (-1)^{|\D(\pi)\setminus
S(\mu)|}
\\
& = \sum\limits_{\pi\in \A_n\cap \LL_\mu^{-1}}
(-1)^{|\D(\pi)\setminus S(\mu)|}.
\end{align*}
Proposition~\ref{arc-mu-unimodal} now completes the proof.
\end{proof}

\section{Further representation-theoretic aspects}\label{sec:remarks}

\subsection{An affine Weyl group action}\label{section-action}

Recall that the affine Weyl group $\tC_{n-2}$ is generated by
$$
S=\{s_0,s_1,\ldots, s_{n-2}\}
$$
subject to the Coxeter relations
\begin{align*}
s_i^2=1 \qquad & \forall i,\\
(s_i s_j)^2=1 \qquad &  \mbox{for }|j-i|>1,\\
(s_i s_{i+1})^3=1 \qquad &  \mbox{for }1\le i<n-3,\\
(s_i s_{i+1})^4=1 \qquad &  \mbox{for }i=0,n-3.
\end{align*}

We now describe a natural action of the group $\tC_{n-2}$ on the set of arc permutations
$\A_n$. Recall that $\sigma_i$ denotes the adjacent transposition $(i,i+1)$.

\begin{definition}\label{c-action}
For every $0\le i\le n-2$, define a map $\rho_i: \A_n\to \A_n $ as
follows:
$$
\rho_i(\pi)=\begin{cases}
\pi \sigma_{i+1},&\hbox{\rm if }\pi \sigma_{i+1}\in \A_n;\\
\pi, &\hbox{\rm otherwise.}
\end{cases}
$$
\end{definition}

\begin{proposition}\label{t.action}
The maps $\rho_i,\ 0\le i\le n-2$, when extended multiplicatively,
determine a well-defined $\widetilde C_{n-2}$-action on
the set of arc permutations $\A_n$.
\end{proposition}

\begin{proof}
To prove that the operation is a well-defined $\tC_{n-2}$-action, it
suffices to show that it is consistent with the defining Coxeter
relations of $\tC_{n-2}$. For every $i$ and
$\pi\in \A_n$,
we have $\rho_i^2(\pi)=\rho_i(\pi)=\pi$ if $\pi \sigma_{i+1}\not\in \A_n$,
and $\rho_i^2(\pi)=\rho_i(\pi\sigma_{i+1})=\pi\sigma_{i+1}^2=\pi$ otherwise. Also,
if $|i-j|>1$, then $\rho_i$ and $\rho_j$ commute, so $(\rho_i\rho_j)^2=1$.

To verify the other two braid relations recall the encoding $\psi:\A_n
\rightarrow \{0,1,\dots,n-1\}\times \{0,1\}^{n-2}$ from
Subsection~\ref{X-diameter}.
For every $1\le i\le n-3$, if $\psi(\pi)_i=\psi(\pi)_{i+1}$ then $\pi
\sigma_{i+1}\not\in \A_n$, thus $\rho_i(\pi)=\pi$; if
$\psi(\pi)_i\ne \psi(\pi)_{i+1}$ then $\pi \sigma_{i+1}\in \A_n$,
thus $\rho_i(\pi)=\pi\sigma_{i+1}$. One concludes that, in both
cases, the effect of $\rho_i$ on $\psi(\pi)$ is to switch the entries $\psi(\pi)_i$ and $\psi(\pi)_{i+1}$.
It follows that for every $1\le i<n-3$, $\rho_i\rho_{i+1}\rho_i= \rho_{i+1}\rho_i\rho_{i+1}$.

Finally, note that for every $\pi\in \A_n$, $\pi \sigma_{n-1}\in \A_n$, thus
$\rho_{n-2}(\pi)=\pi \sigma_{n-1}$ and  $$\psi(\rho_{n-2}(\pi))_{n-2}=1-\psi(\pi)_{n-2},$$ that is,
$\rho_{n-2}$ flips the value of $\psi(\pi)_{n-2}$. On the other hand, as shown above, $\rho_{n-3}$ switches the entries $\psi(\pi)_{n-3}$ and $\psi(\pi)_{n-2}$. One concludes that
$(\rho_{n-3}\rho_{n-2})^4=1$. Since right multiplication by $w_0=n\dots 21$ is
an involution on $\A_n$ and $\rho_i(\pi)=\rho_{n-2-i}(\pi w_0) w_0$, it follows by symmetry that $(\rho_{0}\rho_{1})^4=1$ as well.
\end{proof}

Given a group $G$ with a generating set $R$, and
an action of $G$ on a set ${\cal C}$,
the associated Schreier graph is the graph with vertex set
${\cal C}$ and edge set $\{(x,rx):\ x\in {\cal C}, r\in R\}$.
Recall the graph $X_n$ from Section~\ref{X-graph}. The following is clear by definition.

\begin{remark}\label{isomorphism}
The graph $X_n$ is isomorphic (up to loops) to the Schreier graph
determined by the above $\tC_{n-2}$-action on $\A_n$.
\end{remark}

\begin{corollary}
The above $\widetilde C_{n-2}$-action on $\A_n$ is transitive.
\end{corollary}

\begin{proof}
An action is transitive if and only if the associated Schreier
graph is connected. The result now follows from
Theorem~\ref{A-graph}(i) together with
Remark~\ref{isomorphism}.
\end{proof}

Let $J:=S\setminus \{s_{0}\}=\{s_1,\ldots, s_{n-3}, s_{n-2}\}$.
Recall that the maximal parabolic subgroup $W_J$
is isomorphic to the hyperoctahedral group $B_{n-2}$.

\begin{proposition}
\begin{itemize}
\item[(i)] The maps $\rho_i,\ 0< i\le n-2$, when extended
multiplicatively,
determine a well-defined $B_{n-2}$-action on $\A_n$.

\item[(ii)] The orbits of this action are $\{\pi\in \A_n:\ \pi(1)=k\}$, for $1\le k\le n$.

\item[(iii)] The $B_{n-2}$-action on each of these orbits is
multiplicity-free.
\end{itemize}
\end{proposition}

\begin{proof}
Part $(i)$ follows from Proposition~\ref{t.action}.

For $i>0$, it is clear that $\rho_i(\pi)(1)=\pi(1)$, hence the sets of arc
permutations with fixed first letter are invariant under this
$B_{n-2}$ action. On the other hand, for each $1\le k\le n$, the map $\psi$
defined in Subsection~\ref{X-diameter} determines a bijection from
 $\{\pi\in \A_n:\ \pi(1)=k\}$
to $0{-}1$ vectors of length $n-2$. The restricted $B_{n-2}$ action
on $\{\pi\in \A_n:\ \pi(1)=k\}$ may thus be identified with the
natural $B_{n-2}$-action on all subsets of $[n-2]$, which
is transitive, implying $(ii)$.

To prove $(iii)$ recall that the $B_{n-2}$-representation induced
from the trivial representation of $S_{n-2}$ is multiplicity-free,
see e.g.~\cite[Lemma 2.2(a)]{AFR-hooks},  and notice that this
representation is isomorphic to the $B_{n-2}$-action on all
subsets of $[n-2]$.
\end{proof}

The following question was posed by David Vogan~\cite{Vogan}.

\begin{question} Is the $\widetilde C_{n-2}$-module determined by its action on $\A_n$
multiplicity-free?
\end{question}

\subsection{Representation-theoretic proofs}\label{rep_aspects-section}

By Remarks~\ref{knuth1} and~\ref{knuth2}, the sets $\LL_n$ and $\U_n$ are unions of Knuth classes, hence
they carry the associated symmetric group representation. As noted below Remark~\ref{knuth3}, the set $\Z_n$ of non-unimodal arc permutations is not a
union of Knuth classes. However,  by
Proposition~\ref{non-unimodal-enumeration}, its size is equal to the
number of standard Young tableaux of hook shape plus one box. Here
is a short representation-theoretic proof of
Proposition~\ref{non-unimodal-enumeration}.

\begin{proof}[Proof of Proposition~\ref{non-unimodal-enumeration}]
By the decomposition of the induced exterior algebra $\wedge
V\uparrow^{\S_n}$ into $\S_n$-irreducible characters, which is
described in (\ref{decomposition}),
$$
f^{(n)}+f^{(1^n)}+2\sum\limits_{k=2}^{n-1}
f^{(k,1^{n-k})}+\sum\limits_{k=2}^{n-2} f^{(k,2,1^{n-k-2})}=\dim
\wedge V\uparrow^{\S_n}=n \cdot \dim \wedge V= n 2^{n-2}=|\A_n|.
$$
On the other hand, by  Remark~\ref{knuth2},
$|\U_n|=f^{(n)}+f^{(1^n)}+2\sum\limits_{k=2}^{n-1}
f^{(k,1^{n-k})}$. Since $\U_n\subseteq \A_n$, one concludes
$$
|\Z_n|=|\A_n\setminus \U_n|=\sum\limits_{k=2}^{n-2}
f^{(k,2,1^{n-k-2})}.
$$
\end{proof}

\begin{question}
Find representation-theoretic proofs of Theorem~\ref{thm:TZ} and
other results in Section~\ref{equid}.
\end{question}

By Corollary~\ref{mc_enumeration}, the number of maximal chains in
any interval of $\W(\U_n)$ is equal to twice the number of
standard Young tableaux of shifted staircase shape. It is well
known that this is a
dimension of a projective $\S_n$ representation.

\begin{question}
Determine a projective $\S_n$ representation on the set of
maximal chains in $\W(\U_n)$.
\end{question}

\section{Appendix: shuffles}\label{sec:appendix}

The purpose of this section is to point out that permutations
obtained as shuffles of two increasing sequences have properties similar to those of unimodal and arc permutations.
In analogy to Theorem~\ref{prefixes-theorem} for $\LL_n$, shuffles are obtained as partial fillings of certain shapes.
As a consequence, the weak order restricted to these shuffles has properties analogous to those given in Section~\ref{section-order} for $\W(\U_n)$.

\subsection{Prefixes associated with a rectangle}

Similarly to the partial fillings of the shifted staircase in
Section~\ref{section-prefixes}, we will now consider partial
fillings of a $k\times m$ rectangle with rows labeled
$k,\dots,2,1$ from top to bottom, and columns labeled
$k+1,k+2,\dots,k+m$ from left to right. Again, to each of the
entries $\ell$ in the partial filling, if $\ell$ lies in row $i$
and column $j$, we associate the transposition $(i,j)$. The
product of these transpositions gives a permutation, and the set
of permutations obtained in this way is denoted by $\Sh_{k,m}$.

It is easy to see that $\Sh_{k,m}$ is also the set of permutations
$\pi\in\S_{m+k}$ which are a shuffle of the two sequences
$1,2,\dots,k$ and $k+1,k+2,\dots,k+m$ (that is, both appear as
subsequences of $\pi$ from left to right).

As is the case for $\LL_n$, $\U_n$ and $\A_n$, the set
$\bar\Sh_n=\bigcup_{k+m=n} \Sh_{k,m}$ can be characterized in
terms of pattern avoidance.

\begin{proposition}
$\bar\Sh_n=\S_n(321,2143,2413)$.
\end{proposition}

Shuffles can be easily enumerated, obtaining that
$|\S_n(321,2143,2413)|=|\bar\Sh_n|=2^n-n$ for $n\ge2$. As in
Section~\ref{pattern-section}, it is also the case here that
shuffles are a grid class, consisting of those permutations that
can be drawn on the picture in Figure~\ref{fig:grid_shuffles}. We
write $$\bar\Sh_n=\G_n\left(\begin{array}{c} 1 \\  1
\end{array}\right).$$

\begin{figure}[htb]
\begin{center}
\includegraphics[width=3cm,angle=-90]{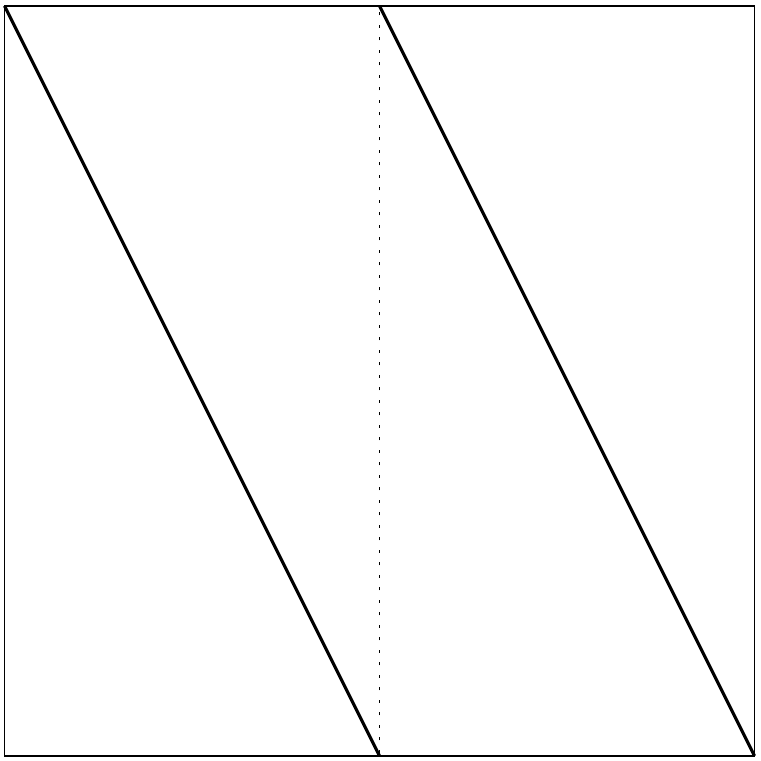}
\caption{\label{fig:grid_shuffles} Grid for shuffles.}
\end{center}
\end{figure}

Shuffles can be characterized as those permutations $\pi$ with
$\D(\pi^{-1})=\{k\}$ for some $k$. For $\pi\in\bar\Sh_n$, if
$\RSK(\pi)=(P,Q)$, then $P$ is a two-row tableau with consecutive
entries $k+1,k+2,\dots,k+\ell$ in the second row, and $Q$ is any
two-row tableau with the same shape as $P$. It follows that
shuffles are a union of Knuth classes.

\subsection{Weak order and enumeration of maximal chains}

Let $\W(\bar\Sh_n)$ be the subposet of $\W(\S_n)$ induced by the
subset $\bar\Sh_n=\S_n(321,2143,2413)$. The following result
follows from arguments analogous to the ones used in
Section~\ref{section-order}.

\begin{proposition} The poset $\W(\bar\Sh_n)$ has the following properties.
\begin{enumerate}\renewcommand{\labelenumi}{(\roman{enumi})}
\item[(i)] The local maxima are exactly the permutations
$\pi_k:=(k+1)(k+2)\dots n12\dots k$ for some $k$.

\item A permutation is in the interval $[e, \pi_k]$ if and only if
it belongs to $\Sh_{k,n-k}$; hence, the number of elements in this
interval is $\binom{n}{k}$.

\item The number of maximal chains in $[e, \pi_k]$ is equal to the
number of standard Young tableaux of rectangular shape $k\times
(n-k)$.
\end{enumerate}
\end{proposition}

\end{document}